\documentclass[]{amsart}
\usepackage{amsmath,amsfonts,amssymb,amsthm}
\usepackage{enumerate}
\usepackage[usenames]{color}

\DeclareMathOperator{\Exp}{Exp}

\DeclareMathOperator{\Diff}{Diff}
\DeclareMathOperator{\Real}{Re}
\DeclareMathOperator{\Fix}{Fix}
\DeclareMathOperator{\Id}{Id}
\DeclareMathOperator{\id}{id}
\DeclareMathOperator{\diag}{diag}

\DeclareMathOperator{\Der}{Der}
\DeclareMathOperator{\Spec}{Spec}
\DeclareMathOperator{\codim}{codim}

\def\dif#1{\Diff_{#1}}
\def\diff#1{\widehat{\Diff}_{#1}}
\def\N{\mathbb N}
\def\R{\mathbb R}
\def\C{\mathbb C}

\def\Cd#1{(\C^#1,0)}

\def\mc#1{{\mathcal #1}}
 
 \newcommand{\yy}{{\bf y}}
 \newcommand{\zz}{{\bf z}}
  \newcommand{\ww}{{\bf w}}

  \newcommand{\vp}{\varphi}
  \newcommand{\wt}[1]{{\widetilde{#1}}}
   \newcommand{\wh}[1]{{\widehat{#1}}}

\newcommand{\eps}{\varepsilon}
\newcommand{\g}{\gamma}
\newcommand{\G}{\Gamma}

\renewcommand{\wt}[1]{\widetilde{#1}}
\renewcommand{\wh}[1]{\widehat{#1}}

   \theoremstyle{plain}
\newtheorem{theorem}{Theorem}[section]
\newtheorem{lemma}[theorem]{Lemma}
\newtheorem{definition}[theorem]{Definition}

\newtheorem{proposition}[theorem]{Proposition}
\theoremstyle{definition}
\newtheorem{remark}[theorem]{Remark}

\theoremstyle{plain}
\newtheorem{Theorem}{Theorem}

 \newcommand{\obra}[3]{{\sc #1} {\em #2}. {#3}.}

\author{L. L\'{o}pez-Hernanz}
\address{Departamento de Matem\'{a}tica\\
Universidade Federal de Minas Gerais\\
Belo Horizonte, Brazil} \email{lorena@mat.ufmg.br}
\thanks{First author partially supported by CNPq, Brazil, process 479134/2013-8. Both authors partially supported by Ministerio de Educaci\'{o}n y Cultura, Spain, process MTM2010-15471, and by Programa Hispano-Brasile\~{n}o de Cooperaci\'{o}n Interuniversitaria, process PHB2010-0122-PC}
\author{F. Sanz S\'{a}nchez}
\address{Departamento de \'{A}lgebra, An\'{a}lisis Matem\'{a}tico, Geometr\'{\i}a y Topolog\'{\i}a\\
Universidad de Valladolid, Spain} \email{fsanz@agt.uva.es}
\thanks{}
\date{}
\title[Parabolic curves asymptotic to formal invariant curves]
{Parabolic curves of diffeomorphisms asymptotic to formal
invariant curves}
\begin{document}
\begin{abstract}
We prove that if $F$ is a tangent to the identity diffeomorphism at $0\in\C^2$ and $\G$ is a formal invariant curve of $F$ then there exists a parabolic curve (attracting or repelling) of $F$ asymptotic to $\G$. The result is a consequence of a more general one in arbitrary dimension, where we prove the existence of parabolic curves of a tangent to the identity diffeomorphism $F$ at $0\in\C^n$ asymptotic to a given formal invariant curve under some additional conditions, expressed in terms of a reduction of $F$ to a special normal form by means of blow-ups and ramifications along the formal curve.
\end{abstract}
\maketitle

\section{Introduction}

In this paper we prove the following theorem.

\begin{Theorem}\label{Theorem1}
Let $F$ be a tangent to the
identity diffeomorphism in $\Cd2$. Given a formal invariant curve $\Gamma$ of $F$ not contained in the set of
fixed points, there exists at least one (attracting or
repelling) parabolic curve for $F$ asymptotic to $\Gamma$.
\end{Theorem}

So far, this result was only known when $\G$ is of ``Briot-Bouquet type'' as we explain below.

A formal curve $\Gamma$ at $0\in\C^2$ is a prime ideal of
$\C[[x,y]]$ generated by a single irreducible formal series. It is invariant for $F$ if $g\circ
F\in\Gamma$ for any $g\in\Gamma$. An attracting parabolic curve
for $F$ is an injective holomorphic map $\varphi:\Delta\to\C^2$,
where $\Delta\subset\C$ is a simply connected domain with
$0\in\partial\Delta$, which is continuous at $0$ with
$\varphi(0)=0$ and such that $\varphi(\Delta)$ is invariant and
attracted to the origin under the iteration of $F$. A repelling parabolic curve for
$F$ is an attracting parabolic curve for $F^{-1}$. We say that a
parabolic curve $\vp$ is asymptotic to $\Gamma$ if there is an
irreducible formal parametrization $\g(s)\in\C[[s]]^2$ of $\G$
such that $\vp$ has $\g(s)$ as asymptotic expansion in $\Delta$ at
$s=0$. In particular, $\vp$ can be blown-up following the
infinitely near points of $\Gamma$.

When $\Gamma$ is convergent, up to reduction of singularities of
$\G$, the existence of parabolic curves asymptotic to $\G$ can be
reduced to the case of dimension 1, where it is always guaranteed
by Leau and Fatou Flower Theorem (see \cite{Lea, Fat}).

The existence of parabolic curves asymptotic to $\Gamma$ when
$\Gamma$ is of ``Briot-Bouquet type'' (its tangent line
corresponds to a non-zero eigenvalue after reduction of
singularities of the infinitesimal generator of $F$) follows from \'{E}calle \cite{Eca} and Hakim
\cite{Hak} (see also Abate et al. \cite{Aba,Aba-B-T} and Brochero
et al. \cite{Bro-C-L}). In fact, the first author has proved in \cite{Lop}
that these parabolic curves are the sums (in the sense of Ramis
summability theory) of $\Gamma$, so they can be recovered from
$\Gamma$ by means of the Borel-Laplace summation process. Notice that Camacho and Sad Theorem \cite{Cam-S} guarantees the
existence of at least one formal curve of Briot-Bouquet type.

Invariant curves of Briot-Bouquet type of $F$ are
convergent when the infinitesimal generator of $F$ is convergent,
but in general this is not the case. In fact, there are examples
of diffeomorphisms in $\Cd 2$ without analytic invariant curve
(see \cite{Rib}). On the other hand,
formal invariant curves that are not of Briot-Bouquet type appear in every
saddle-node singularity of the infinitesimal generator whose weak
direction (the one corresponding to the zero eigenvalue) is
transversal to the exceptional divisor. For example, Euler's
vector field has a purely formal invariant curve which is not of
Briot-Bouquet type.

We prove Theorem \ref{Theorem1} as a consequence of a more general
result in arbitrary dimension $n$, which we explain briefly now.

One of our main tools is the reduction of the diffeomorphism to what we will call
{\em Ramis-Sibuya form} along $\G$: assuming that
$\G$ is not contained in the set of fixed points of $F$, there is a
process consisting in finitely many local blow-ups and ramifications, following the infinitely near points of the curve $\Gamma$, which
transforms $F$ into a diffeomorphism which can be written as
\begin{equation}\label{eq:RS-in-intro}
\wt{F}(x,\yy)=\left(x+x^{k+p+1}(1+o(1)),\,
\yy+x^k\left[(D(x)+x^pC)\yy+O(x^{p+1})\right]\right)
\end{equation}
where $(x,\yy)\in\C\times\C^{n-1}$, $k\ge1$ and either $p=0$ and $C\neq0$ or
$p\geq 1$, $D(x)$ is a diagonal matrix whose entries are
polynomials of degree at most $p-1$ with $D(0)\neq 0$ and $C$ is a
constant matrix which commutes with $D(x)$. Moreover, $\G$ is transformed
into a formal curve $\wt{\G}$, invariant for $\wt{F}$,
non-singular and transversal to the hypersurface $\{x=0\}$, the
set of fixed points of $\wt{F}$. To the infinitesimal
generator of a diffeomorphism $\wt{F}$ as in
(\ref{eq:RS-in-intro}) we can associate a system of $n-1$ formal
meromorphic ODEs which is, essentially, the starting point of
Ramis-Sibuya's paper \cite{Ram-S}, and this is the reason for our name. In fact, we
obtain a reduction of $F$ to Ramis-Sibuya form by transforming its
infinitesimal generator $\log F$ into a vector field to which we
can associate, by means of blow-ups and
ramifications, a system as in Ramis-Sibuya paper, and whose exponential will be a diffeomorphism as in (\ref{eq:RS-in-intro}). In order for this approach to work, we need to prove
first that blow-ups and ramifications are compatible with the
operation of taking the infinitesimal generator. This compatibility,
as well as the compatibility with the property of existence of
parabolic curves asymptotic to $\G$, is the content of
section~\ref{sec:blow-ups}.

In the two-dimensional case $n=2$, the reduction of a
diffeomorphism $F$ to Ramis-Sibuya form is already established in
the literature (see \cite{Aba, Bro-C-L}) and can be seen as a consequence of the well known result on reduction of singularities of vector fields (see for instance Seidenberg
\cite{Sei}).

In dimension $n\geq 3$, the reduction to Ramis-Sibuya form is not a completely unknown result either.
In fact, after some initial blow-ups at
infinitely near points of $\G$ that transform the vector field
$\log F$ into one to which we can associate a system of $n-1$
ODEs, the essential argument is an adaptation to the non-linear
case of a well known result in the theory of systems of linear
meromorphic ODEs, due to Turrittin \cite{Tur}. However, we have
included a complete proof of the result in
section~\ref{sec:reduction} since we could not find any statement
in the precise terms that we need for our purposes: for instance, we show that the centers of the blow-ups satisfy an additional property, besides the usual invariance, that guarantees the compatibility with the operation of taking the
infinitesimal generator (see the
definition of {\em permissible center} in section 5). Note
that the infinitesimal generator of a diffeomorphism as in
(\ref{eq:RS-in-intro}) has in particular a non-nilpotent linear
part and hence a reduction to Ramis-Sibuya form contains a weak
form (since we make use of ramifications) of a {\em local
uniformization} of $\log F$ {\em along} $\G$ (see
\cite{Can-M-R,Pan}).

Once we have reduction to Ramis-Sibuya form, it suffices to study the problem of existence of parabolic curves asymptotic to $\G$ for diffeomorphisms $F$ as in (\ref{eq:RS-in-intro}).
The case $p=0$ corresponds, in any
dimension, to what we have called the Briot-Bouquet case in
dimension two. In this case \'{E}calle \cite{Eca} and Hakim \cite{Hak} prove that there
exists at least one parabolic curve of $F$ asymptotic to
$\G$. We generalize their result to the case $p\geq 1$
(i.e. the eigenvalue corresponding to the tangent direction of
$\G$ vanishes) assuming an extra condition on $F$. When
$p\geq 1$ put $D(x)=\diag(d_2(x),...,d_n(x))$ and define the {\em
saddle domain} of $F$ (relatively to $\G$) as the intersection
of the ``saddle domains'' of all the linear
ODEs $x^{p+1}y'=d_j(x)y$, for $j=2,...,n$. In
section~\ref{sec:main} we give details of this definition and we
prove the following result (see Theorem~\ref{th:main-technical}): if the saddle domain is non-empty and contains one of the $(k+p)$-th
roots of $-1$ then $F$ has a parabolic curve asymptotic to
$\G$. We obtain this parabolic curve as a fixed point of a contraction map in a convenient Banach space which, for the Briot-Bouquet case, coincides with the one used by Hakim in \cite{Hak}.

Combining this result with the mentioned compatibility of the reduction to Ramis-Sibuya form with the  existence of parabolic curves, we prove the following result. Given a
diffeomorphism $F$ with an invariant formal curve $\G$ not
contained in the set of fixed points, we say that $\G$ is {\em
well placed for $F$} if, after reduction to Ramis-Sibuya form
(\ref{eq:RS-in-intro}), either we are in the Briot-Bouquet case $p=0$
or $p\geq 1$ and the
saddle domain contains one of the $(k+p)$-th roots of $-1$.

\begin{Theorem}\label{Theorem2} Let $F\in\dif1\Cd n$ and let $\G$ be an invariant formal curve of $F$
not contained in the set of fixed points of $F$. If $\Gamma$ is
well placed for $F$, then there exists at least one attracting
parabolic curve for $F$ asymptotic to $\Gamma$.
\end{Theorem}

In the two-dimensional case $n=2$, we prove in
section~\ref{sec:dim2} that any formal invariant curve $\Gamma$ of
$F$ is well placed either for $F$ or for the inverse
diffeomorphism $F^{-1}$. Theorem~\ref{Theorem1} will then follow
from Theorem~\ref{Theorem2}.

We conclude the introduction with a final remark. Theorem~\ref{Theorem2} establishes a sufficient condition on $F$ and $\G$ in order that a parabolic curve of $F$ asymptotic to $\G$
exists. To check such a condition we need to perform the process
of reduction to Ramis-Sibuya form, which depends both on $F$ and on
$\G$. However, it is worth to mention that such a process only
involves a finite part of the Taylor expansion of $F$ at the origin and hence the condition of $\G$ being well placed for $F$
  is \emph{finitely determined}.

\section{Parabolic curves and formal invariant curves}\label{sec:definitions}

Denote by $\dif1\Cd n$ the group of germs of holomorphic
diffeomorphisms $F:\Cd n\to\Cd n$ which are tangent to the
identity, that is, $DF(0)=\Id$. Any $F\in\dif1\Cd n$ determines an
endomorphism $\mathcal{F}:g\mapsto g\circ F$ of the local algebra
$\mathcal{O}_n$ of germs of holomorphic functions at the origin.
Conversely, any endomorphism of $\mathcal{O}_n$
 inducing the identity
 in the cotangent space
$\frak{m}/\frak{m}^2$, where $\frak{m}$ is the maximal ideal,
determines an element of $\dif1\Cd n$. Replacing $\mathcal{O}_n$
by its formal completion $\wh{\mathcal{O}}_n$, we may analogously
consider the group $\diff1\Cd n$ of formal tangent to the identity
diffeomorphisms. If $F\in\diff1\Cd n$, its {\em order} is defined
as the minimum $r$ such that
$(\mathcal{F}-\id)(\frak{m})\subset\frak{m}^r\setminus\frak{m}^{r+1}$.

 To any
formal tangent to the identity diffeomorphism $F$, we can associate its {\em infinitesimal
generator}: the unique formal vector field $X\in
\Der_\C(\wh{\mathcal{O}}_n)$ which satisfies, for any
$g\in\wh{\mathcal{O}}_n$,
\begin{equation}\label{eq:exponential}
\mathcal{F}(g)=g\circ F=\sum_{j=0}^\infty \frac{{X}^j(g)}{j!}
\end{equation}
where $X^0(g)=g$ and $X^j(g)=X(X^{j-1}(g))$ for $j\geq 1$, or equivalently
\begin{equation}\label{eq:logarithm}
X(g)=\sum_{j=1}^\infty\frac{(-1)^{j+1}}{j}(\mathcal{F}-\id)^j(g).
\end{equation}
Its
{\em multiplicity} at $0$, defined as $\nu_0(X)=\min\{l\,:
\,X(\frak{m})\subset\frak{m}^l\setminus\frak{m}^{l+1}\}$,
 coincides with the order of
$F$. Thus $\nu_0(X)\geq 2$. Reciprocally, each $X\in
\Der_\C(\wh{\mathcal{O}}_n)$ with $\nu_0(X)\geq 2$ is the
infinitesimal generator of a unique formal diffeomorphism
$F\in\diff1\Cd n$ satisfying (\ref{eq:exponential}). If $X$ is the
infinitesimal generator of $F$, we write either $X=\log F$ or
$F=\Exp X$. In general, the infinitesimal generator of
$F\in\dif1\Cd n$ is not convergent (i.e. it does not belong to
$\Der_\C(\mathcal{O}_n)$), although if $X$ is convergent
then $\Exp X\in\dif1\Cd n$.

Once we choose coordinates $\zz=(z_1,\ldots,z_n)$ at $0\in\C^n$,
a formal diffeomorphism $F\in\diff1\Cd n$ is written as $ F(\zz)=(z_1\circ
F,...,z_n\circ F)=(z_1+f_1(\zz),\ldots,z_n+f_n(\zz))$, where
$f_j(\zz)\in\C[[\zz]]$ and $ \log
F=a_1(\zz)\frac{\partial}{\partial
z_1}+\cdots+a_n(\zz)\frac{\partial}{\partial z_n}$ with
$a_j(\zz)\in\C[[\zz]]$. The order of $F$ is then the minimum of
the orders of the series $f_j(\zz)$ for $j=1,...,n$. Moreover, if
$k=\nu_0(X)$, the homogeneous components of degree $k$ of the
series $a_j(\zz)$ and of $f_j(\zz)$ are the same for $j=1,...,n$.
We also have, using identities (\ref{eq:exponential}) and (\ref{eq:logarithm}), that the ideal of $\wh{\mathcal{O}}_n$ generated by $\{f_1, f_2,..., f_n\}$ is equal to the ideal generated by $\{a_1, a_2,..., a_n\}$. Thus, in the case of $F$ being holomorphic, its set of fixed points $\Fix(F)$ coincides with the singular locus of $X$ (which is then an analytic set even when $X$ is formal).

\strut

 A {\em formal curve $\G$ at $0\in\C^n$} is by
definition a prime ideal of the ring $\wh{\mathcal{O}}_n$ such
that the quotient ring $\wh{\mathcal{O}}_n/\G$ has dimension one.
Once we fix
analytic coordinates $\zz$ at the origin, a formal curve $\G$ is
determined univocally by a {\em formal parametrization}: a tuple
$\g(s)=(\g_1(s),...,\g_n(s))\in\C[[s]]^n$ so that $\g_j(0)=0$ and
$g(\g(s))\equiv 0$ if and only if $g\in\G$. The \emph{multiplicity} of $\G$, denoted by $\nu_0(\G)$, is the
minimum of the orders of the components of any {\em irreducible}
parametrization, i.e., a parametrization $\g(s)$ which cannot be
written as $\g(s)=\sigma(s^l)$ with $\sigma(s)$ another
parametrization of $\G$ and $l>1$. The curve
$\G$ is called {\em non-singular} if $\nu_0(\G)=1$.
The {\em tangent line} of $\G$ is, by definition, the complex line
$T\G$ at $0\in\C^n$ whose projective coordinates are equal to
$[\g_1(s)/s^m,\ldots,\g_n(s)/s^m]\mid_{s=0}$ for $\g(s)$ an
irreducible parametrization of $\G$ and $m=\nu_0(\G)$. For an
appropriate choice of coordinates, we may always consider a parametrization of $\G$ of the form
$\g(s)=(s^m,\g_2(s),...,\g_n(s)), $ where $m=\nu_0(\G)$ and each $\g_j(s)$ has order
at least $m$ (thus $T\G$ is transversal to $\{z_1=0\}$). Such a parametrization is called a {\em Puiseux parametrization} of $\G$.

Consider $F\in\dif1\Cd n$, $X=\log F$ and let $\G$ be a formal curve. We say that $\G$ is {\em invariant for $F$} if one of the following equivalent conditions holds:
\begin{itemize}
\item[1.] $g\circ F\in\G$ for any $g\in\G$.
\item[2.] $X(g)\in\G$ for any $g\in\G$.
\item[3.] For any parametrization $\g(s)$ of $\G$ there exists $h(s)\in\C[[s]]$ such that $
X\mid_{\g(s)}=h(s)\g'(s)$, where
$\g'(s)\in\C[[s]]^n$ is the vector of formal derivatives of the
components of $\g(s)$.
\end{itemize}

The equivalence between the two first conditions follows from identities (\ref{eq:exponential}) and (\ref{eq:logarithm}) (see \cite{Rib} for a proof in dimension $n=2$). The equivalence between the two last ones is quite well known in differential algebra; a possible proof is the following: since the equivalence is preserved under blow-ups at infinitely near points of $\G$, using reduction of singularities of $\G$ we can assume that it is non-singular, and the equivalence is then easy to prove.
Notice that in the last condition the series $h(s)$ is different from zero if and only if $\G$ is not contained in the singular locus of $X$, which is equivalent to saying that $\G$ is not contained in $\Fix(F)$.
Observe also that, since $\log F^{-1}=-\log F$, a formal curve $\G$ is
invariant for $F$ if and only if it is invariant for $F^{-1}$.

It is worth to remark that, while a planar diffeomorphism
$F\in\dif1\Cd 2$ always has a formal invariant curve (by Camacho
and Sad's Theorem \cite{Cam-S}), this is not the case for $n>2$ (see the
example of G\'omez-Mont and Luengo \cite{Gom-L}). On the other
hand, formal invariant curves of diffeomorphisms, even in
dimension two, do not need to be convergent (see for instance an
example of Rib\'on \cite{Rib}). A formal non-convergent invariant
curve $\G$ has no geometric meaning a priori. We look for the
possibility of finding a ``geometric realization'' of $\G$:
roughly speaking, an analytic object attached to $\G$ and
invariant for $F$. A more precise statement of the problem is to
find a parabolic curve of $F$ asymptotic to $\G$.
\begin{definition}\label{def:parabolic-curve}
 An {\em attracting parabolic curve for $F\in\dif1\Cd n$} is an injective
holomorphic map $\vp:\Delta\to\C^n$ where $\Delta$ is a simply
connected domain in $\C$ with $0\in\partial\Delta$ such that $\varphi(\Delta)$ is contained in the domain of a representant of $F$ and which
satisfies:
\begin{enumerate}
\item[1.]  $\varphi$ is continuous at the origin putting $\varphi(0)=0$
(continuity). \item[2.] $F(\varphi(\Delta))\subset \varphi(\Delta)$
(invariance). \item[3.]
$\lim_{j\to\infty}\varphi^{-1}(F^j(\varphi(s)))=0\ \text{ for all }
s\in\Delta$ (stability).
\end{enumerate}

A {\em repelling parabolic curve of $F$} is an attracting
parabolic curve of the inverse diffeomorphism $F^{-1}$.
\end{definition}

\begin{definition}\label{def:asymptotic}
A parabolic curve $\vp:\Delta\to\C^n$ will be called {\em asymptotic
to the formal curve $\G$} if there exists a parametrization
$\g(s)\in\C[[s]]^n$ of $\G$ in some coordinates $\zz$ of $\C^n$ such
that $\zz\circ\vp$ has $\g(s)$ as asymptotic expansion at $s=0$ in
$\Delta$; i.e., for any $N\in\N$, there exist constants
$c_N,\varepsilon_N>0$ such that
$$
\|\zz\circ\vp(s)-J_N\g(s)\|\leq c_N |s|^{N+1}\;\;\forall
s\in\Delta\mbox{ with }|s|\leq\varepsilon_N,
$$
where $J_N\g$ denotes the vector of truncations of the components
of $\g$ up to order $N$, i.e. $\|\g(s)-J_N\g(s)\|=o(s^N)$.
\end{definition}
A similar definition appears in \cite{Aba-T} under the name ``robust parabolic curve''.

\strut

 If the components of a parametrization of $\G$ were
multisummable (in the sense of Malgrange and Ramis
\cite{Ram,Mal-R}), its multisum $\vp:\Delta\to\C^n$, defined in
some open sector $\Delta$, could be a good candidate to solve the
problem of finding a parabolic curve asymptotic to $\G$. However,
as far as we know, multisummability of $\G$ is an open question in
general, except for the case where $\log F$ is a holomorphic
vector field (see Braaksma \cite{Bra}). We should also mention the
work of L\'{o}pez \cite{Lop}, where it is proved that $\G$ is
summable (and its sum gives a parabolic curve) in the case $n=2$
if $\G$ is of Briot-Bouquet type (definition below), and the work
of Brochero and L\'{o}pez \cite{Bro-L} where the Gevrey character
of the infinitesimal generator $\log F$ is proved, thus implying
the Gevrey character of $\G$ (see \cite{Can}).

Instead of searching for multisummability of $\G$, we look for
parabolic curves asymptotic to invariant formal curves by an
approach similar to the one in Abate \cite{Aba} and
Brochero et al. \cite{Bro-C-L}: we first analyze, in
section~\ref{sec:main}, the case where $F$ can be written in a
particular reduced form, called Ramis-Sibuya form, and then, in section~\ref{sec:reduction},
we prove that we can reduce $F$ to Ramis-Sibuya form by means of blow-ups and ramifications. Reduction to Ramis-Sibuya form in the case of two-dimensional diffeomorphisms is
considerably easier and comes from the reduction of singularities of planar vector fields. We consider this case separately in section~\ref{sec:dim2}, where we prove Theorem~\ref{Theorem1}.

\section{Parabolic curves for diffeomorphisms in Ramis-Sibuya
form}\label{sec:main}

Throughout this section, we consider $F\in\dif1\Cd n$ having a
formal invariant non-singular curve $\G$.

For convenience, a system of local analytic coordinates
$(x,\yy=(y_2,...,y_n))$ at $0\in\C^n$ will be called {\em adapted
for $\G$} if $\G$ is transversal to the hypersurface $\{x=0\}$, thus having in those coordinates a parametrization
as a graph $\g(x)=(x,\g_2(x),...,\g_n(x))\in\C[[x]]^n$. The {\em
order of contact} of $\G$ with the $x$-axis is the minimum of the
orders of the components $\g_j(x)$.

\begin{definition}\label{def:RS-form}
We say that the pair $(F,\G)$ is in {\em Ramis-Sibuya
form} (\emph{RS-form} for short) if there exist adapted
coordinates $(x,\yy)$ for $\G$ at $0\in\C^n$ for which $F$ is
written as
\begin{equation}\label{eq:RS-form}
F(x,\yy)=\left(x+\lambda x^{k+p+1}(1+\psi(x,\yy)),\,\,
\yy+x^k\left[b(x)+(D(x)+x^pC+x^{p+1}A(x))\yy+O(\|\yy\|^2)\right]\right),
\end{equation}
where $k\geq 1$, $p\geq 0$, $\lambda\in\C^*$, $\psi\in\C\{x,\yy\}$
with $\psi(0)=0$, $b(x)\in\C\{x\}^{n-1}$ with $b(0)=0$, $A(x)$ is
a matrix with entries in $\C\{x\}$ and, moreover,
\begin{enumerate}[(i)]
\item either $p=0$ and $C\neq 0$ (this case is called {\em
Briot-Bouquet case}) or
\item $p\geq 1$, $D(x)$ is a diagonal matrix of polynomials of degree at most $p-1$ with $D(0)\neq 0$ and $D(x)$ commutes
 with $C$.
 \end{enumerate}
Such adapted coordinates $(x,\yy)$ will be called {\em
RS-coordinates} for $(F,\G)$. The matrix $D(x)+x^pC$ is called the
{\em principal linear part} of $(F,\G)$ in the coordinates
$(x,\yy)$.
\end{definition}

Notice that both $k$ and $p$ do not depend on the choice of
RS-coordinates. In fact $k+1$ is the order of the diffeomorphism
$F$ and $k+p+1$ is the order of the restriction
$F|_\G\in\diff1(\C,0)$. The number $p$ is called the {\em Poincar\'{e}
rank} of $(F,\G)$. Notice also that the set of fixed points of $F$ is equal to $\{x=0\}$, since $\lambda\neq 0$, and then $\G$ is not contained in $\Fix(F)$.

The infinitesimal generator of
$F$ as in (\ref{eq:RS-form}) is written as
\begin{equation}\label{eq:X-RS-form}
X=\log F=x^k\left[x^{p+1}u(x,\yy)\frac{\partial}{\partial
x}+\left(c(x)+(\mathcal{D}(x)+x^p\mathcal{C}+x^{p+1}\mathcal{A}(x))\yy+O(\|\yy\|^2)\right)\frac{\partial}{\partial\yy}
\right]
\end{equation}
where $u(0,0)=\lambda$, $c(0)=0$ and either $p=0$ and $\mathcal{C}\neq0$ or $p\ge1$ and $\mathcal{D}(x)$ is a diagonal matrix of polynomials of degree at most $p-1$ which commutes with $\mathcal{C}$ and such that $\mathcal{D}(0)\neq 0$. Reciprocally, if $X$ is a vector field as in (\ref{eq:X-RS-form}) satisfying the conditions above, then the diffeomorphism $\Exp X$ is in Ramis-Sibuya form. The matrix $\mathcal{D}(x)+x^p\mathcal{C}$ is called the
{\em principal linear part} of $X$ in the coordinates $(x,\yy)$. To a vector field of the form
(\ref{eq:X-RS-form}) we can associate the system of $n-1$ formal ODEs
$$x^{p+1}u(x,\yy)\yy'=c(x)+(\mc D(x)+x^p\mc C+x^{p+1}\mc A(x))\yy+O(\|\yy\|^2),$$
which has irregular singular point at $x=0$ and Poincar\'{e} rank equal to $p$. The properties assumed for the principal linear part $\mc D(x)+x^p\mc C$ are essentially those
considered in the work of Ramis and Sibuya \cite{Ram-S}, where
multisummability of the formal solutions is proved in the case where the coefficients of the system are convergent. This is the
reason of our choice of the name ``Ramis-Sibuya  form'' in
Definition~\ref{def:RS-form}.

\begin{remark}\label{rk:RS-form} Assume that $(F,\G)$ is in RS-form, written
as in (\ref{eq:RS-form}) in RS-coordinates $(x,\yy)$.
\begin{enumerate}[(a)]
\item Let $l$ denote the order of contact of $\G$ with the
$x$-axis. Then any component of the vector $b(x)\in\C\{x\}^{n-1}$
has order at least $l$. \item If $(\bar{x},\bar{\yy})$ are new
coordinates at $0\in\C^n$ obtained from $(x,\yy)$ by a change of
variables of the form
 $$
  \bar{x}=h(x),\;\;\bar{\yy}=\yy+\Lambda(x,\yy),\;\;
  \mbox{where }h'(0)\neq
  0,\frac{\partial\Lambda}{\partial\yy}(x,0)=O(x^{p+1})
$$
  then $(\bar{x},\bar{\yy})$ are also RS-coordinates. Moreover, if $h(x)-x=O(x^{p+1})$
   then the principal linear parts
  of $(F,\G)$ in the coordinates $(x,\yy)$ and $(\bar{x},\bar{\yy})$ are the same.

\end{enumerate}
\end{remark}

Now we deal with the problem of finding a parabolic curve of $F$
asymptotic to $\G$ when $(F,\G)$ is in RS-form. With the notations
of Definition~\ref{def:RS-form}, we define the {\em attracting
directions} of $(F,\G)$ (with respect to the coordinates
$(x,\yy)$) as the $k+p$ half-lines $\R_+\xi$, where $\xi^{k+p}=-\lambda$.
If $\G$ is convergent, the restriction $F|_\G\in\dif1(\C,0)$ has order $k+p+1$
and, by Leau and Fatou's theorem, any such direction is the
bisectrix of an attracting petal of $F|_\G$, thus giving rise to
$k+p$ attracting parabolic curves of $F$ defined in sectors
bisected by these directions and asymptotic to $\G$. When $\G$ is formal,
\'{E}calle \cite{Eca} and Hakim \cite{Hak} obtain the same conclusion in the
Briot-Bouquet case. Here, we extend this result to the non Briot-Bouquet case ($p\geq 1$),
proving the existence of a parabolic curve attached to a fixed attracting
direction $\tau$ if some extra hypothesis (concerning $\tau$) is
satisfied. Our arguments are based mainly on Hakim's paper;
however, since our statement is not exactly the same as Hakim's
one, our result is stated for $(F,\G)$ irrespective of the value
of the Poincar\'{e} rank $p$.

Let $D(x)+x^pC$ be the principal linear part of $(F,\G)$
with respect to coordinates $(x,\yy)$. Write $D(x)=\diag(d_2(x),...,d_n(x))$ where each
$d_j(x)$ is a polynomial of degree at most $p-1$. For each $j$
such that $d_j(x)\not\equiv 0$, we denote by $\nu_j$ the order of
$d_j$ at $0$ and write $ d_j(x)=d_{j0}x^{\nu_j}+\cdots$ with
$d_{j0}\neq 0$. Observe that if $p\geq 1$ then $\nu_j=0$ for at
least one $j$, since $D(0)\neq0$. Notice also that if $\mc{D}(x)+x^p\mc{C}$ is the principal linear part
of $\log F$ then the $j$-th entry of $\mathcal{D}(x)$ is also of the form $d_{j0}x^{\nu_j}+\cdots$ if $d_j(x)\not\equiv 0$, and is equal to zero if $d_j(x)\equiv0$.

Let us define the
 {\em saddle
domain} of $(F,\G)$ (in the coordinates $(x,\yy)$) as the set
\begin{equation*}
V=\mathop{\bigcap_{2\le j\le n}}\limits_{d_j\not\equiv0}
\left\{x\in\C\,:\,\Real\left(-\frac{d_{j0}}{\lambda
x^{p-\nu_j}}\right)>0\right\}\subset\C.
\end{equation*}
We notice that the saddle domain $V$ is the whole plane $\C$ in
the Briot-Bouquet case $p=0$, and a union of proper open sectors
at the origin in the case $p\geq 1$. Observe that $V$ is non-empty if $n=2$. Also, if $p\geq 1$ and $D(0)$ is invertible then $V\neq\emptyset$ implies that the eigenvalues of $D(0)$ belong to the so called ``Poincar\'{e} domain'', i.e. their convex hull in the plane does not contain $0$.
The name ``saddle domain'' comes from the
following fact: if $p\geq 1$, the linear system $
x^{p+1}Y'=\mathcal{\overline{D}}(x)Y$ of ODEs, where $\mathcal{\overline{D}}(x)$ is the diagonal submatrix of $\mc{D}(x)$ formed by its nonzero entries, has $Y=0$ as the unique
solution which is bounded in closed subsectors of $V$, while the
other ones are unbounded along any direction contained in $V$.

We are now ready to state the main result in this section. We use the standard notation
$S(\tau,\alpha,r)$ for the open sector at $0\in\C$ bisected by $\tau$
with opening $\alpha>0$ and radius $r\in\R^+\cup\{\infty\}$.

\begin{theorem}\label{th:main-technical}
Assume that $F\in\dif1\Cd n$ has a formal invariant curve $\G$
such that $(F,\G)$ is in Ramis-Sibuya  form. Let $\tau$ be an
attracting direction and let $V$ be the saddle domain
in some RS-coordinates $(x,\yy)$. If $\tau$ is contained in $V$ then
there exists an attracting parabolic curve for $F$ asymptotic to
$\G$. More precisely, there exist positive numbers
$\eta_0,\delta_0$ with $S(\tau,\eta_0,\delta_0)\subset V$ and a
holomorphic map $\bar{\vp}:S(\tau,\eta_0,\delta_0)\to\C^{n-1}$ such
that

\begin{enumerate}[(i)]
\item  The map $\vp:S(\tau,\eta_0,\delta_0)\to\C^n$ given by
$x\mapsto(x,\bar{\vp}(x))$ in the coordinates $(x,\yy)$ is an attracting parabolic curve for
$F$ asymptotic to $\G$.
\item For any $0<\eta\leq\eta_0$ there
exists $0<\delta\leq\delta_0$ such that, if $\delta'\leq\delta$,
then the restriction of $\vp$ to $S(\tau,\eta,\delta')$ is also an
attracting parabolic curve for $F$.
\item If $W$ is any sector
containing the direction $\tau$ and $\psi:W\to\C^n$ is an attracting parabolic curve of $F$ asymptotic to $\G$, defined by $\psi(x)=(x,\bar\psi(x))$ in the coordinates $(x,\yy)$, then $\vp\equiv\psi$ in $S(\tau,\eta_0,\delta_0)\cap W$.
\end{enumerate}
\end{theorem}

\begin{remark}
Although both the attracting direction $\tau$ and the saddle domain
$V$ depend on the RS-coordinates, the condition $\tau\subset V$
is independent of the coordinates.
\end{remark}

The rest of this section is devoted to the proof of
Theorem~\ref{th:main-technical}.

\strut

{\em Choosing coordinates.-} We start with an expression of $F$ as
in (\ref{eq:RS-form}) in the given coordinates $(x,\yy)$. Up to a
linear change of the form $x\mapsto\alpha x$, we can assume that
$\lambda=-1$. Moreover,
 it is a routine to check that
a polynomial change of variables of the form $ x\mapsto x+P(x)$
with $P(x)=O(x^{p+1})$ permits to consider new RS-coordinates
$(x,\yy)$ for which
$$
  x\circ F\,(x,\yy)=x-x^{k+p+1}+O(x^{k+p+1}\|\yy\|,x^{k+2p+2}).
$$
Notice, using Remark~\ref{rk:RS-form}, that the principal linear
part has not changed. We will need to consider coordinates having
high contact with $\G$. In fact, consider a Puiseux
parametrization $\g(x)=(x,\bar{\g}(x))\in\C[[x]]^n$ of $\G$ in the
coordinates $(x,\yy)$. For any $m\geq 2$, we consider new
coordinates $(x,\yy_m)$ with $\yy_m=\yy-J_{m+p-1}\bar{\g}(x)$ (recall that
$J_{N}\bar{\g}(x)$ denotes the truncation of
$\bar{\g}(x)$ up to order $N$). Then $(x,\yy_m)$ are new
RS-coordinates for which $\G$ has order of contact at least $m+p$
with the $x$-axis. Hence, using Remark~\ref{rk:RS-form}, $F$ is
written as
\begin{equation}\label{eq:F-m}
\begin{array}{rcl}
  x\circ F\,(x,\yy_m)&=&x-x^{k+p+1}+O(x^{k+p+1}\|\yy_m\|,x^{k+2p+2}) \\
  \yy_m\circ F\,(x,\yy_m)&=&\yy_m+x^k(D(x)+x^{p}C)\yy_m+
  O(x^{k+p+1}\|\yy_m\|,x^k\|\yy_m\|^2,x^{k+p+m})
\end{array}
\end{equation}
Notice that the principal linear part of $(F,\G)$ with respect to
$(x,\yy_m)$ does not depend on $m$. Neither do the attracting
directions and the saddle domain $V$, since the coordinate $x$ is
preserved. Write, as above, $D(x)=\diag(d_2(x),...,d_n(x))$
and $d_j(x)=d_{j0}x^{\nu_j}+\cdots$ with $d_{j0}\neq 0$ if
$d_j(x)\not\equiv 0$.

\strut

{\em Reduction to parabolic curves having $m$-th order contact
with $\G$.-} Given an integer $m\geq p+2$ and constants
$\eta,\delta>0$, we consider the Banach space $(\mc
B^m_{\eta,\delta},{\bf n})$ defined by
$$
\mc B^m_{\eta,\delta}=
\left\{u\in\mathcal{O}(S(\tau,\eta,\delta),\C^{n-1})\,: \,{\bf
n}(u)=\sup\Bigl\{\frac{\|u(x)\|}{|x|^{m-1}}:x\in
S(\tau,\eta,\delta)\Bigr\}<\infty\right\}$$  and the closed subspace
$$\mc{H}^m_{\eta,\delta}=\{u\in\mc B^m_{\eta,\delta}: {\bf
n}(u)\leq 1, \|u'(x)\|\le|x|^{m-p-2}\, \forall x\in S(\tau,\eta,\delta)\}.$$

Let $I$ be the interval of points $\eta\in\R$ satisfying
$0<\eta<2\pi/(k+p)$ and
$$
S(\tau,\eta,\infty)\subset V\cap\mathop{\bigcap_{2\le j\le
n}}\limits_{d_j\not\equiv0}
\left\{x\in\C\,:\,\Real\left(d_{j0}x^{k+\nu_j}\right)>0\right\}.
$$
Notice that $I\neq\emptyset$ by the hypothesis on the attracting
direction $\tau$.

Fix an integer $m_0$ such that $m_0\ge\max\{p+2,
p+2-\min\{\Real(\alpha_j)\,:\,\alpha_j\in\Spec(\mathcal{C})\}\}$,
where $\mathcal{D}(x)+x^p\mathcal{C}$ is the principal linear part
of the infinitesimal generator of $F$, with the notations of
(\ref{eq:X-RS-form}).

\begin{proposition}\label{pro:main-technical-m}
For any $m\geq m_0$ and any $\eta\in I$ there exists
$\delta=\delta(m,\eta)>0$, with $S(\tau,\eta,\delta)$
contained in the domain of the
coordinates $(x,\yy_m)$, and for any $\delta'$ with $0<\delta'\le\delta$ there exists a unique $u_{m,\delta'}\in\mc H^m_{\eta,\delta'}$ such that:

\begin{enumerate}[(a)]
\item  The map $\vp_{m,\delta'}:S(\tau,\eta,\delta')\to\C^n$ written in the
coordinates $(x,\yy_m)$ as $\vp_{m,\delta'}:x\mapsto(x,u_{m,\delta'}(x))$ is an
attracting parabolic curve for $F$.
 \item For any
$0<\delta'\le\delta$, we have $u_{m,\delta'}=u_{m,\delta}|_{S(\tau,\eta,\delta')}$.
\item If $W$ is a sector containing $S(\tau,\eta,\delta')$ for some
$\delta'\le\delta$ and $\psi:W\to\C^n$ is an attracting parabolic curve for $F$
written in coordinates $(x,\yy_m)$ as $\psi(x)=(x,v(x))$, with
$v|_{S(\tau,\eta,\delta')}\in\mc{H}^m_{\eta,\delta'}$, then
$u_{m,\delta}\equiv v$ in $S(\tau,\eta,\delta')$.
\end{enumerate}
\end{proposition}

\noindent Before the proof of
Proposition~\ref{pro:main-technical-m}, let us see how it permits to prove Theorem~\ref{th:main-technical}. Apply
Proposition~\ref{pro:main-technical-m} to $m'>m\geq m_0$ and a
fixed constant $\eta_0\in I$. We obtain two parabolic curves $
\vp_m:S(\tau,\eta_0,\delta_m)\to\C^n$,
$\vp_{m'}:S(\tau,\eta_0,\delta_{m'})\to\C^n $ of $F$ which we write
in coordinates $(x,\yy_m)$ and $(x,\yy_{m'})$ respectively as $\vp_m(x)=(x,u_m(x))$, with
$u_m\in\mc{H}^m_{\eta_0,\delta_m}$, and $\vp_{m'}(x)=(x,u_{m'}(x))$, with $u_{m'}\in\mc{H}^{m'}_{\eta_0,\delta_{m'}}$. Let $\rho=\min\{\delta_m,\delta_{m'}\}$ and write the restriction
of $\vp_{m'}$ to $S(\tau,\eta_0,\rho)$ in the coordinates $(x,\yy_m)$
as $\vp_{m'}(x)=(x,v(x))$. Then we have that
$v(x)=P(x)+u_{m'}(x)$, where
$P(x)=J_{m'+p-1}\bar{\g}(x)-J_{m+p-1}\bar{\g}(x)$, a vector of
polynomials of order at least $m+p$ and of degree at most
$m'+p-1$. We obtain that the restriction of $v$ to
$S(\tau,\eta_0,\rho')$ belongs to $\mc{H}^m_{\eta_0,\rho'}$ for some
$0<\rho'\leq\rho$ and, by Proposition~\ref{pro:main-technical-m}, (c), we have $v\equiv u_m$ in $S(\tau,\eta_0,\rho')$ (notice that
$\vp_{m'}\mid_{S(\tau,\eta_0,\rho')}$ is again a parabolic curve by
part (b)). Thus $\vp_m$ and $\vp_{m'}$ coincide in their common
domains and we obtain a holomorphic map
$$
\vp:\bigcup_{m\geq m_0}S(\tau,\eta_0,\delta_m)\to\C^n.
$$
We may assume that $\delta_m\leq\delta_{m_0}$ for any $m\geq m_0$ and thus $\varphi$ is defined in $S(\tau,\eta_0,\delta_0)$ where $\delta_0=\delta_{m_0}$. By construction, $\varphi$ is a parabolic curve of $F$ asymptotic to $\G$, proving part
(i) of Theorem~\ref{th:main-technical}. Part (ii) follows from
properties (b) and (c) of Proposition~\ref{pro:main-technical-m}
stated for $\eta\leq\eta_0$. Finally, let us show property (iii)
of Theorem~\ref{th:main-technical}. The map $
\bar{\psi}_{m_0}(x)=\bar{\psi}(x)-J_{m_0+p-1}\bar{\g}(x)$ has the formal series
$\bar{\g}(x)-J_{m_0+p-1}\bar{\g}(x)$ as asymptotic expansion on $W$ at $0$. Since this series has order at least $m_0+p$, we deduce that
there exist some $\eta,\delta>0$ such that
$S(\tau,\eta,\delta)\subset W$ and the restriction
$\bar{\psi}_{m_0}\mid_{S(\tau,\eta,\delta)}$ belongs to
$\mc{H}^{m_0}_{\eta,\delta}$. On the other hand,
$(x,\bar{\psi}_{m_0}(x))$ is the expression of $\psi$ in the
coordinates $(x,\yy_{m_0})$ and hence, by
Proposition~\ref{pro:main-technical-m}, (b) and (c), up to
considering smaller constants $\eta,\delta$, we obtain that $\psi$
coincides with $\vp$ on $S(\tau,\eta,\delta)$ and thus on $W\cap
S(\tau,\eta_0,\delta_0)$.

\strut

The rest of the paragraph is devoted to the proof of
Proposition~\ref{pro:main-technical-m}, divided in several steps.
We fix some $m\geq m_0$ and $\eta\in I$ and assume that $F$ is
written as in (\ref{eq:F-m}) in the coordinates $(x,\yy_m)$. For
simplicity, from now on we will drop the indexes ``$m,\eta$''. In
particular $\yy=\yy_m$, $\mc{B}_\delta=\mc{B}^m_{\eta,\delta}$,
$\mc{H}_\delta=\mc{H}^m_{\eta,\delta}$ and $
S_\delta=S(\tau,\eta,\delta). $

 Denote by $ f(x,\yy)=x\circ
F,\;\overline{F}(x,\yy)=\yy\circ F $ the series defined in
(\ref{eq:F-m}) and let $\eps>0$ be small enough so that these
series are convergent for any $(x,\yy)$ with $|x|\leq\eps$,
$\|\yy\|\leq\eps^{m-1}$. We look first for an element
$u\in\mc{H}_{\delta}$, with $\delta$ sufficiently small, so that
\begin{equation}\label{eq:u-F1}
u(f(x,u(x)))=\overline{F}(x,u(x))
\end{equation}
Such an $u$ will be constructed as a fixed point of a contracting
map in $\mc{H}_\delta$, in an analogous way to what is done in
Hakim's work \cite{Hak} for the Briot-Bouquet case.

\strut

{\em Some technical results.-} Let $D(x)+x^pC$ and
$\mathcal{D}(x)+x^p\mathcal{C}$ be respectively the
principal linear parts of the diffeomorphism $F$ and of its
infinitesimal generator $X=\log F$ in the RS-coordinates $(x,\yy)$
(with the notations of (\ref{eq:F-m}) and (\ref{eq:X-RS-form})).
Recall that they do not depend on $m$. We have that
\begin{equation}\label{eq:tildeA}
I+x^k(D(x)+x^pC)=J_{k+p}\left(\exp(x^k(\mathcal{D}(x)+x^p\mathcal{C}))\right).
\end{equation}

Define
$$
\wh{E}(x)=\exp\left( -\int\frac{\mathcal{D}(x)}{x^{p+1}}dx
\right);\;\;E(x)=\wh{E}(x)x^{-\mathcal{C}}.
$$

Observe that $\wh{E}(x)$ is a fundamental solution of the linear
system of ODEs $x^{p+1}Y'=-\mathcal{D}(x)Y$, since $\mc{D}(x)$ is
diagonal, whereas $E(x)$ is a fundamental solution of
$x^{p+1}Y'=-(\mc{D}(x)+x^p\mc{C})Y$, since $\mc{D}(x)$ commutes with
$\mc{C}$. The matrix $\wh{E}(x)$ has an essential singularity at
$x=0$ when $p\geq 1$ and $E(x)$ is a multivalued function of
$x\in\C^*$. They have the interesting properties stated
in the following lemma, inspired by Lemma 3.1 in \cite{Rol-S-S}:
\begin{lemma}\label{lm:ExEF1}
 For any $(x,\yy)$ sufficiently small and $x\neq 0$,
we have
\begin{equation*}
\begin{split}
    \wh{E}(x)\wh{E}(f(x,\yy))^{-1}&=\exp(-x^{k}\mc{D}(x))+O(x^{k+p+1},x^k\|\yy\|)\\
    E(x)E(f(x,\yy))^{-1}&=\exp(-x^{k}(\mc{D}(x)+x^p\mc{C}))+O(x^{k+p+1},x^k\|\yy\|)
\end{split}
\end{equation*}
\end{lemma}
\begin{proof}
Let $Z(x)$ be either $E(x)$ or $\wh{E}(x)$, a fundamental solution
of the system $x^{p+1}Y'=-B(x)Y$, where $B(x)$ is either
$\mc{D}(x)+x^p\mc{C}$ or $\mc{D}(x)$, correspondingly. Put
$\Omega(x,z)=Z(x+x^{p+1}z)$. If we fix $x$ and consider $\Omega$
as a function of $z$ then it satisfies the (regular) system
$$\frac{\partial\Omega}{\partial
z}=\frac{-B(x+x^{p+1}z)}{(1+x^{p}z)^{p+1}}\Omega(x,z).$$
On the other hand, we have $\Omega(x,0)=Z(x)$ and thus
$$
\Omega(x,z)=\exp\left(-\int_0^z\frac{B(x+x^{p+1}t)}{(1+x^{p}t)^{p+1}}dt
\right)Z(x)
$$
(using again that $\mc{D}(x)$ is diagonal and commutes with
$\mc{C}$). Hence
$$
Z(x)Z(x+x^{p+1}z)^{-1}=Z(x)\Omega(x,z)^{-1}=\exp\left(
\int_0^z\frac{B(x+x^{p+1}t)}{(1+x^{p}t)^{p+1}}dt \right).
$$
The integrand in the equation above is an analytic function of
$(x,t)$ and may be written as
$\frac{B(x+x^{p+1}t)}{(1+x^{p}t)^{p+1}}=B(x)+O(x^{p}t)$. By
integration, we obtain, for any $x\neq 0$ and $z$ sufficiently
small,
$$
Z(x)Z(x+x^{p+1}z)^{-1}=\exp(zB(x))+x^{p}z^2\Lambda(x,z),
$$
where $\Lambda$ is analytic at the origin. The
result follows using the expression of $f(x,\yy)$.
\end{proof}

\strut

{\em Dynamics in dimension one.-} Given $u\in\mc B_{\eps}$ with
${\bf n}(u)\leq 1$, we consider the holomorphic map
 $f_{u}:S_{\eps}\to\C$ defined by
 $
 f_{u}(x)=f(x,u(x)).
 $

Notice (since $m\geq p+2$) that $|f_u(x)-x+x^{k+p+1}|\leq
B|x|^{k+2p+2}$ for $x\in S_\eps$, with some $B>0$ independent of
$u$. Then, taking into account that $\eta<2\pi/(k+p)$ and arguing
as in the proof of Leau and Fatou flower theorem in \cite{Lor}, we
can prove the following result.

\begin{proposition}\label{fatoufu}
There exists a constant $\delta$ with $0<\delta\leq\eps$ such
that, for every $0<\delta'\le\delta$ and every $u\in\mc
H_{\delta'}$ we have $f_{u}(S_{\delta'})\subset S_{\delta'}$.
Moreover, there is a constant $c>0$ such that, if
$x_0\in S_{\delta}$ and we denote by $x_j=f_{u}(x_{j-1})$, for
${j\geq 1}$, the orbit of $x_0$ by $f_{u}$, then for all $j\in\N$
we have
$$
x_j^{k+p}\sim \frac1{(k+p)j}\quad\mbox{ and }\quad
|x_j|^{k+p}\le c\frac{|x_0|^{k+p}}{1+j|x_0|^{k+p}}
$$
(where $w_j\sim z_j$ means that $\lim_{j\to\infty} w_j/z_j=1$).

\end{proposition}

We also need the following lemma.
\begin{lemma}\label{lm:sum-of-xj}Put $\sigma=\min\{\Real(\alpha_j):\alpha_j\in\Spec(\mathcal{C})\}$. With the same notations as in Proposition~\ref{fatoufu}, and if
$\delta$ is sufficiently small, we have:
\begin{enumerate}[(i)]
\item For any real number $s>k+p$, there exists a constant $K_s>0$ such that $
\sum_{j\geq 0}|x_j|^s\leq K_s|x_0|^{s-k-p}$.
\item There exist constants $M_1, M_2>0$ such that
$\|(x_j/x_0)^{\mathcal{C}}\|\le M_1|x_j/x_0|^{\sigma-1/2}$ and
$\|\wh{E}(x_0)\wh{E}(x_j)^{-1}\|\leq M_2$ for any $x_0\in
S_\delta$ and any $j\geq 0$.
\end{enumerate}
\end{lemma}
\begin{proof}
Part (i) is proved without difficulty using the second inequality
in Proposition~\ref{fatoufu} (see \cite{Hak}). For the first
inequality of (ii), assuming $\mc{C}=U+N$ is in Jordan form ($U$
diagonal and $N$ nilpotent), there exists a constant $M$,
depending only on $\mc{C}$, such that
$$
\|z^\mc{C}\|\leq M|z|^{-1/2}\|z^U\|\leq M|z|^{\sigma-1/2},\;\mbox{
for }0<|z|\leq 1.
$$
The result follows putting $z=\frac{1}{c^{1/(k+p)}}\frac{x_j}{x_0}$, where $c$ is the constant
of Proposition~\ref{fatoufu}. Let us prove the
second inequality of (ii). Using Lemma~\ref{lm:ExEF1}, the
definition of $\mc{H}_\delta$ and the fact that $m\geq p+2$, we
may write
$$\wh{E}(x_0)\wh{E}(x_1)^{-1}=\exp(-x_{0}^{k}\mc{D}(x_0))+\theta_u(x_0)$$
where $\|\theta_u(x_0)\|\leq K|x_0|^{k+p+1}$ for any $x_0\in
S_\delta$ and any $u\in\mc{H}_\delta$, with some $K>0$ independent
of $u$. From equation (\ref{eq:tildeA}), we have
$\mathcal{D}(x)=\diag(b_2(x),...,b_n(x))$ where each
$b_i(x)\in\C[x]$ has degree at most $p-1$ and satisfies
$b_i(x)=d_{i0}x^{\nu_i}+O(x^{\nu_i+1})$ if
$d_i\not\equiv 0$ or $b_i\equiv 0$ otherwise. By the definition of the interval
$I$, $\Real(d_{i0}x_l^{k+\nu_i})>0$  for all $l\ge0$ if
$b_i(x)\not\equiv0$ (recall that $x_l\in S_\delta$ for all $l\ge0$). We obtain
$$\|\wh{E}(x_0)\wh{E}(x_j)^{-1}\|\leq \prod_{l=0}^{j-1} (1+K|x_l|^{k+p+1})\leq \prod_{l=0}^{\infty} (1+K|x_l|^{k+p+1})$$
for $\delta$ sufficiently small. The result follows from the fact that $x_j^{k+p}\sim\frac1{(k+p)j}$
and a classical criterium
of convergence of infinite products.
\end{proof}

\strut

{\em The contraction map.-} Define
$H(x,\yy)=\yy-E(x)E(f(x,\yy))^{-1}\overline{F}(x,\yy)\in\C\{x,\yy\}$.
Using Lemma~\ref{lm:ExEF1} and identity
(\ref{eq:tildeA}), we have
\begin{equation}\label{eq:H}
H(x,\yy)=O(x^{k+p+1}\|\yy\|,x^k\|\yy\|^2,x^{k+p+m}).
\end{equation}

\begin{proposition}\label{pro:contraction-map}
If $\delta>0$ is sufficiently small and we put $x_j=f_{u}(x_{j-1})$ for ${j\geq 1}$ for $u\in\mc H_\delta$ and $x_0\in S_\delta$, the series
$$
Tu(x_0)=\sum_{j\ge0}E(x_0)E(x_j)^{-1}H(x_j,u(x_j))
$$
is normally convergent and defines a contracting map $T:u\mapsto
Tu$ from $\mc H_{\delta}$ to itself. Moreover, $u\in\mc{H}_\delta$
is a fixed point of $T$ if and only if $u$ satisfies equation
(\ref{eq:u-F1}).
\end{proposition}

\begin{remark} In the Briot-Bouquet case $p=0$ we have from (\ref{eq:tildeA}) that $\mathcal{D}(x)\equiv0$ and $\mc{C}=C$, so the operator $T$ and the space $\mc{H}_{\delta}$
are exactly the ones considered by Hakim in \cite{Hak}.
\end{remark}

\begin{proof}
By (\ref{eq:H}), there exists a constant $K>0$ such that, if
$\delta$ is small enough and $x\in S_\delta$ then
$\|H(x,u(x))\|\leq K|x|^{k+p+m}$ independently of $u\in\mc
H_{\delta}$ (recall that we have assumed that $m\geq p+2$). Thus,
using Lemma~\ref{lm:sum-of-xj} (ii),
\begin{equation}\label{eq:l1}
\|E(x_0)E(x_j)^{-1}H(x_j,u(x_j))\|\leq M_1M_2K\frac{|x_j|^{\sigma-\frac{1}{2}+k+p+m}}{|x_0|^{\sigma-\frac{1}{2}}},
\end{equation}
and using Lemma~\ref{lm:sum-of-xj}, (i), the series $Tu(x_0)$ is
normally convergent for any $x_0\in S_\delta$ and $u\in\mc
H_\delta$, defines a holomorphic map
$Tu\in\mathcal{O}(S_\delta,\C^{n-1})$ and satisfies
$Tu(x)=O(|x|^{m})$. Then, $\|Tu(x)\|\leq|x|^{m-1}$ for $|x|<\delta$,
if $\delta$ is small enough. We compute the derivative:
$$\begin{array}{rcl}
(Tu)'(x_0)& = & \displaystyle
\sum_{j\ge0}E'(x_0)E(x_j)^{-1}H(x_j,u(x_j))\\[5pt]
&-&
\displaystyle\sum_{j\ge0}E(x_0)E(x_j)^{-1}x'_jE'(x_j)E(x_j)^{-1}H(x_j,u(x_j))\\[5pt]
&+& \displaystyle \sum_{j\ge0}E(x_0)E(x_j)^{-1}x'_j\dfrac{\partial
H}{\partial
x}(x_j,u(x_j))\\[5pt]
&+&\displaystyle\sum_{j\ge0}E(x_0)E(x_j)^{-1}x'_j\dfrac{\partial H}{\partial
\yy} (x_j,u(x_j))u'(x_j).
\end{array}
$$
where the notation $x'_j$ stands for the derivative, at the point
$x_0$, of $f_u^j(x)$ as a function of $x$. Arguing as in
\cite{Hak}, Lemma 4.4, we have that $|x'_j|\le 2|x_j/x_0|^{k+p+1}$
for all $j\ge0$. On the other hand, recall that
$E'(x)=-\frac{\mathcal{D}(x)+\mathcal{C}x^p}{x^{p+1}} E(x)$. Using
these two arguments, equations (\ref{eq:H}), (\ref{eq:l1}), Lemma
\ref{lm:sum-of-xj} and the inequality
$\|u'(x_j)\|\leq|x_j|^{m-p-2}$ from the definition of $\mc
H_\delta$, it is easy (but tedious) to check that each of the four
summands in the expression of $(Tu)'(x_0)$ is bounded by a
positive constant times $|x_0|^{m-p-1}$. Therefore,
$\|(Tu)'(x)\|\le|x|^{m-p-2}$ for $|x|<\delta$, if $\delta$ is
small enough, which shows that $Tu\in\mc H_\delta$.

Let us prove that $T$ is a contraction map. Take $u, v\in\mc
H_{\delta}$. For $x_0\in S_\delta$, put $x_j=f_u^j(x_0)$ and
$y_j=f_v^j(x_0)$. Arguing as in \cite{Hak}, Lemma 4.9, with a slight modification as in
\cite{Lop}, Lemma 4.4, if $\delta$ is small enough, there exists $M>0$, independent of $x_0\in S_\delta$ and $u,v\in\mc H_\delta$,
such that
\begin{equation}\label{eq:yj-xj}
|y_j-x_j|\le M|x_0|^m{\bf n}(v-u).
\end{equation}
Write $Tu(x_0)-Tv(x_0)=U_1+U_2+U_3$, where
$$\begin{array}{rcl}
U_1&=&\displaystyle \sum_{j\ge0}E(x_0)E(x_j)^{-1}\left[H(x_j,u(x_j))-H(y_j,v(y_j))\right]\\[5pt]
U_2&=&\displaystyle
\sum_{j\ge0}x_0^{-\mathcal{C}}x_j^{\mathcal{C}}\left[\wh{E}(x_0)\wh{E}(x_j)^{-1}-\wh{E}(x_0)\wh{E}(y_j)^{-1}\right]
H(y_j,v(y_j))\\[5pt]
U_3&=&\displaystyle
\sum_{j\ge0}x_0^\mathcal{-C}x_j^{\mathcal{C}}\left[I-y_j^{\mathcal{C}}x_j^{-\mathcal{C}}\right]\wh{E}(x_0)\wh{E}(y_j)^{-1}H(y_j,v(y_j)).
\end{array}$$

To bound $U_1$, we write first $$
H(x_j,u(x_j))-H(y_j,v(y_j))=H(x_j,u(x_j))-H(y_j,u(x_j))
+H(y_j,u(x_j))-H(y_j,v(y_j)). $$
 Using
(\ref{eq:H}) and taking into account that $u\in\mc{H}_\delta$,
that $m\geq p+2$ and that $x_j\sim y_j$
when $j\to\infty$ (see Proposition~\ref{fatoufu}), we obtain that
there are constants $c_1,c_2>0$ (independent of $\delta,x_0,u,v$)
such that
$$
\begin{array}{rcl}
\|H(x_j,u(x_j))-H(y_j,u(x_j))\|&\leq&
c_1|x_j|^{k+p+m-1}|x_j-y_j|,\\[5pt]
 \|H(y_j,u(x_j))-H(y_j,v(y_j))\|&\leq&
c_2|x_j|^{k+p+1}\|v(y_j)-u(x_j)\|.
\end{array}
$$
On the other hand, using similar arguments, there exists $c_3>0$
such that
$$
\begin{array}{rcl}
\|v(y_j)-u(x_j)\|&\le&\|v(y_j)-v(x_j)\|+\|v(x_j)-u(x_j)\|\\[5pt]
&\leq&c_3\left(|x_j|^{m-p-2}|y_j-x_j|+|x_j|^{m-1}{\bf
n}(v-u)\right).
\end{array}
$$
Thus, by Lemma \ref{lm:sum-of-xj} and equation (\ref{eq:yj-xj}),
and using the fact that $m\ge p+2-\sigma$, we get finally that
there exists $B_1>0$, such that, if $\delta>0$ is small enough and
$x_0\in S_\delta$ and $u,v\in\mc{H}_\delta$ then
$$
\|U_1\|\leq B_1|x_0|^m{\bf n}(v-u).
$$

To bound $U_2$, put $\wh{E}(x)=\exp(R(x))$ with
$$
R(x)=-\int\frac{\mathcal{D}(x)}{x^{p+1}}dx=\diag\left(r_2(x),...,r_n(x)\right),$$
and $r_l(x)=\frac{d_{l0}}{(p-\nu_l)x^{p-\nu_l}}\tilde r_l(x)$,
where $\tilde r_l(x)\in\C[x]$ is a polynomial of degree at most
$p-\nu_l-1$ such that $\tilde r_l(0)=1$, for $l=2,...,n$.
Observe that $d_{l0}\neq0$ if and only if $r_l(x)\not\equiv0$. Since
$x_1=x_0-x_0^{k+p+1}+O(x_0^{k+p+1}u(x_0),x_0^{k+2p+2})$, we have
$$r_l(x_1)=r_l(x_0)+d_{l0}x_0^{k+\nu_l}(1+\theta_u(x_0)),\;l=2,...,n,$$
where $\|\theta_u(x_0)\|\leq c_4|x_0|$ for some constant $c_4>0$
independent of $u$. Since $\Real(d_{l0}x_1^{k+\nu_l})>0$ if
$r_l(x)\not\equiv0$ (by the definition of the interval $I$ and by Proposition~\ref{fatoufu}), we obtain, by induction on $j\geq 1$, that
$\Real(r_l(x_0)-r_l(x_j))\le0$ and $\Real(r_l(x_0)-r_l(y_j))\le0$
for $|x_0|\leq\delta$ sufficiently small and for all $2\le l\le n$
and all $j\ge1$. Therefore, if we put
$$
\wh{E}(x_0)\wh{E}(x_j)^{-1}-\wh{E}(x_0)\wh{E}(y_j)^{-1}
=\exp(a)-\exp(b),
$$
where $a=R(x_0)-R(x_j)$, $b=R(x_0)-R(y_j)$, then all diagonal entries of $a$ and $b$ are in the half space $\{z:\Real(z)\leq 0\}$ and thus
$$
\|\wh{E}(x_0)\wh{E}(x_j)^{-1}-\wh{E}(x_0)\wh{E}(y_j)^{-1}\|\leq|a-b|\max_{\xi\in[a,b]}\|\exp\xi\|\le
\|R(y_j)-R(x_j)\|.
$$
Finally, since $R'(x)=O(|x|^{-p-1})$ for $x\in S_\delta$, we
conclude, using $x_j\sim y_j$, that there exists a constant $c_5>0$ independent of
$\delta,x_0,u,v$ such that
$$\left\|\wh{E}(x_0)\wh{E}(x_j)^{-1}-\wh{E}(x_0)\wh{E}(y_j)^{-1}\right\|
\le c_5\frac{|y_j-x_j|}{|x_j|^{p+1}}$$ for all $j$. Then, applying Lemma
\ref{lm:sum-of-xj} and equations (\ref{eq:H}) and (\ref{eq:yj-xj}),
there exists $B_2>0$ such that, if $\delta>0$ is small enough and
$x_0\in S_\delta$ and $u,v\in\mc{H}_\delta$ then
$$\|U_2\|\le B_2|x_0|^m{\bf n}(u-v).$$

To bound $U_3$, we use
$$\left\|I-y_j^{\mathcal{C}}x_j^{-\mathcal{C}}\right\|=
\left\|I-\exp\left(\mathcal{C}\log\frac{y_j}{x_j}\right)\right\|\le
c_6\frac{|y_j-x_j|}{|x_j|},$$ with some constant $c_6>0$. Again by
Lemma \ref{lm:sum-of-xj}, equations (\ref{eq:H}) and (\ref{eq:yj-xj}) and the fact than $x_j\sim y_j$, there
exists $B_3>0$ such that, if $\delta>0$ is small enough and
$x_0\in S_\delta$ and $u,v\in\mc{H}_\delta$ then
$$\|U_3\|\le B_3|x_0|^m{\bf n}(u-v).$$

Summarizing, we have for any $u,v\in\mc{H}_\delta$ and every $x\in
S_\delta$,
$$\|Tu(x)-Tv(x)\|\le (B_1+B_2+B_3)|x|^m{\bf n}(u-v).
$$ We conclude that $T:\mc
H_{\delta}\to\mc H_{\delta}$ is a contraction map if $\delta$ is
small enough.

Finally, notice that we can rewrite the map $T$ as
$$
\begin{array}{rcl}
Tu(x_0)&=&E(x_0)\sum_{j\geq
0}\left(E(x_j)^{-1}u(x_j)-E(x_{j+1})^{-1}\overline{F}(x_j,u(x_j))\right)\\[5pt]
&=&
u(x_0)-E(x_0)E(x_1)^{-1}\overline{F}(x_0,u(x_0))+E(x_0)E(x_1)^{-1}Tu(x_1).
\end{array}
$$
The first equality shows that if $u\in\mc{H}_\delta$ satisfies
(\ref{eq:u-F1}) then $u$ is a fixed point of $T$. The second one
shows the converse.
\end{proof}

\strut

{\em End of the proof of Proposition~\ref{pro:main-technical-m}}.-
Fix $m\geq m_0$ and $\eta\in I$. Applying Proposition~\ref{pro:contraction-map}, there exists some $\delta=\delta(m,\eta)>0$
and a unique $u_{m,\delta'}\in\mc H^m_{\eta,\delta'}$ for any $0<\delta'\le\delta$ which is a solution of the invariance equation (\ref{eq:u-F1}). Moreover, by Proposition~\ref{fatoufu}, the orbit of any point by the map $f_{u_{m,\delta'}}$ converges to $0$. This shows statement (a) of Proposition~\ref{pro:main-technical-m}. Statement (b) follows once we remark that, if $\delta'<\delta$ and $u\in\mc H^m_{\eta,\delta}$ satisfies equation (\ref{eq:u-F1}),
then the restriction of $u$ to $S(\tau,\eta,\delta')$ belongs to $\mc
H^m_{\eta,\delta'}$ and satisfies also equation (\ref{eq:u-F1}) by Proposition~\ref{fatoufu}. Let us prove finally statement (c). By Proposition~\ref{fatoufu}, if we put $\tilde{v}=v|_{S(\tau,\eta,\delta')}$ then $f_{\tilde{v}}(S(\tau,\eta,\delta'))\subset S(\tau,\eta,\delta')$ and, since $\psi(x)=(x,v(x))$ is a parabolic curve of $F$, $\tilde{v}$ satisfies equation (\ref{eq:u-F1}). By Proposition~\ref{pro:contraction-map}, we have then $\tilde{v}=u_{m,\delta'}$.

\section{Two-dimensional case}\label{sec:dim2}
Consider $F\in\dif1\Cd 2$ and let $\G$ be a formal invariant curve of
$F$ which is not contained in the set of fixed points of $F$. In
this paragraph we show how to use the results in the preceding
paragraph in order to prove Theorem~\ref{Theorem1}. The proof
consists in two steps. In the first one, we reduce the general
situation to the case where $(F,\G)$ is in Ramis-Sibuya
form. This reduction, which can already be found in
\cite{Aba,Bro-C-L} and that we briefly report below, will also be
obtained as a particular case of a general result in any dimension which we discuss in section~\ref{sec:reduction}
(see Theorem~\ref{th:reduction-singularities-F}). In the second
step, we prove that if $(F,\G)$ is in RS-form then
Theorem~\ref{th:main-technical} applies either for $F$ or for
$F^{-1}$ and we find a parabolic curve of $F$ (attracting or
repelling) asymptotic to $\G$.

The reduction to RS-form is not much more than the well known result
of {\em reduction of singularities} of the formal vector field
$X=\log F$ {\em along} $\G$ (see \cite{Sei} or the book
\cite{Can-C-D} for a more recent presentation): after a finite
composition $\Phi:(\C^2,0)\to(\C^2,0)$ of local blow-ups centered at
the infinitely near points of $\G$, the transform of $X$ is written
as $\Phi^*X=x^k \bar{X}$, where $\{x=0\}$ is the exceptional divisor of
$\Phi$, $\bar{X}(0)=0$ and the linear part $D_0\bar{X}$ is
non-nilpotent; moreover, the strict transform $\wt{\G}$ of $\G$ by
$\Phi$ is non-singular, transversal to $\{x=0\}$ and invariant for
$\Phi^*X$. Since $\nu_0(\Phi^*X)\geq\nu_0(X)$ and the operation of
taking the infinitesimal generator is compatible with blow-ups at
singular points (see \cite{Bro-C-L}), $\Phi^*X$ is the infinitesimal
generator of the transformed diffeomorphism $\wt{F}\in\dif1\Cd 2$
satisfying $\Phi\circ\wt{F}=F\circ\Phi$. From the fact that $D_0\bar{X}$ is non-nilpotent, and up to several additional blow-ups, we deduce that $\wt{F}$ is in RS-form. Finally, if $\wt{\vp}:\Delta\to\C^2$
is a parabolic curve of $\wt{F}$ asymptotic to $\wt{\G}$ then
$\vp=\Phi\circ\wt{\vp}$ is a parabolic curve of $F$ asymptotic to
$\G$.

Concerning the second step, we can, in fact, give a better
description of how many parabolic curves a diffeomorphism in
RS-form may have:
\begin{proposition}\label{pro:RS-dim2}
Assume that $(F,\G)$ is in RS-form and let $k+1$ be the order of
$F$ and $p$ the Poincar\'{e} rank of $(F,\G)$. Then we
have:
\begin{enumerate}[(a)]
    \item If $p=0$ then there are  $k$ attracting and $k$
    repelling
    parabolic curves of $F$ asymptotic to $\G$.
    \item If $1\leq p<k$ then there are $p$ attracting and $p$
    repelling
    parabolic curves of $F$ asymptotic to $\G$.
    \item If $k<p$ there is at least one attracting and one repelling parabolic curve
    of $F$ asymptotic to $\G$.
\item If $1\leq p=k$ then there are $p$ attracting or $p$
repelling
    parabolic curves
    of $F$ asymptotic to $\G$.
    \end{enumerate}
\end{proposition}

\begin{proof}
Statement (a) corresponds to the Briot-Bouquet case and is already
proved in \cite{Hak}. It is also an immediate consequence of
Theorem~\ref{th:main-technical} since in this case the saddle
domain is the whole complex plane $\C$. Let us assume that $p\geq
1$.  Choose RS-coordinates $(x,y)$ at $0\in\C^2$ such that $F$ is
written as
$$
F(x,y)=\left(x-x^{k+p+1}(1+o(1)),y+x^k(ay+O(x)+O(y^2))\right)
$$
where $a\in\C^*$. These coordinates are also RS-coordinates for
the inverse diffeomorphism $F^{-1}$, which is written as
$$
F^{-1}(x,y)=\left(x+x^{k+p+1}(1+o(1)),y+x^k(-ay+O(x)+O(y^2))\right).
$$
Thus, both $F$ and $F^{-1}$ have the same saddle domain
$V=\{\text{Re}(\frac{a}{x^p})>0\}$ in these coordinates, whereas
the attracting directions for $F$ (respectively for $F^{-1}$) are determined by
the $(k+p)$-th roots of $1$ (respectively of $-1$). Notice also
that we only need to prove the statement for attracting parabolic
curves in the cases (b) and (c). Write $V=V_1\cup\cdots\cup V_p$
as a union of disjoint open sectors of opening $\pi/p$. If $p<k$,
since the angle between two consecutive attracting directions is
$2\pi/(k+p)<\pi/p$, each $V_j$ contains at least one attracting
direction. Theorem~\ref{th:main-technical} thus implies (b). Let
us show (c) by contradiction. Assume that $p>k$ and there is no
attracting direction contained in $V$. Then, since
$2\pi/(p+k)>\pi/p$, the sector between any pair of consecutive
attracting directions should contain one of the sectors $V_j$,
showing that there are as much sectors $V_j$ as attracting
directions, which is impossible.
 Finally, if $p=k$ the angle between two consecutive attracting directions is $\pi/p$. Then,
either each $V_j$ contains exactly one such attracting direction
or the attracting directions are in the boundary of the sectors
$V_j$. But, in this last case, the bisectrices of the sectors
$V_j$ are attracting directions of $F^{-1}$. This shows (d), using Theorem~\ref{th:main-technical}.
\end{proof}

\section{Blow-ups and ramifications of diffeomorphisms}\label{sec:blow-ups}

A consequence of the property of a parabolic curve $\vp$ being
asymptotic to a formal curve $\G$ is that these two objects share
the same {\em sequence of iterated tangents}. This is a very
natural notion (partially motivated by the corresponding concept
for trajectories of real vector fields developed in
\cite{Can-Mou-San1,Can-Mou-San2}) which means that, for any sequence
of finitely many punctual blow-ups $\pi:\widetilde{\C}^n\to\C^n$,
the transformed map
$\tilde{\varphi}=\pi^{-1}\circ\varphi:\Delta\to\wt{\C}^n$ has a
limit $\lim_{s\to 0}\tilde{\varphi}(s)\in\pi^{-1}(0)$, and this
limit is equal to the corresponding point of the transformed of
 $\G$ by $\pi$.

This property suggests that the operation of blowing-up is a good
tool for the problem of existence of parabolic curves of
$F\in\dif1\Cd n$ asymptotic to a formal invariant curve $\G$. This
approach works satisfactorily in dimension two, as we have shown
in section~\ref{sec:dim2}, since we have reduction of
singularities for vector fields and punctual blow-ups are
compatible with the operation of taking the infinitesimal
generator of $F$.

 In higher
dimension, the situation is much less clear since we do not have
reduction of singularities of vector fields by blowing-up points
in the sequence of iterated tangents of an invariant formal curve;
not even by blowing-up invariant centers of bigger dimension:
there are examples of (polynomial) vector fields in $\C^3$ with a
formal invariant curve $\G$ at the origin whose strict transform
by any sequence of blow-ups of points or invariant analytic
non-singular curves at the corresponding point of the transform of
$\G$ has always a nilpotent linear part (see Example 1.2 in
\cite{Pan}). Together with blow-ups, we also need to consider
ramifications.

\strut

Let $Z$ be a germ of a non-singular analytic submanifold of $\C^n$
at the origin. Let $U$ be an open neighborhood of $0\in\C^n$ where
$Z$ has a representative $Z'\subset U$ which is a non-singular
analytic submanifold of $U$. We will denote by $\pi_Z:M\to U$ the
blow-up of $U$ with center $Z'$, which we will simply call {\em
the blow-up with center $Z$}. The fiber $\pi_Z^{-1}(0)$ can be
identified with the projective space over the normal space,
$T_0\C^n/T_0Z$, of $Z$ at the origin.

Let $X$ be a formal vector field at $0\in\C^n$ with $\nu_0(X)\ge1$. A germ $Z$ of non-singular
analytic submanifold is {\em invariant for $X$} if
$X(g)\in I(Z)$ for any $g\in I(Z)$, where $I(Z)$ denotes the ideal of holomorphic germs
vanishing on $Z$. In this case, the
 {\em multiplicity of $X$ along
$Z$} is defined as
$$
\nu_Z(X)=\min\{l\geq 1\,:\,X(I(Z))\subset
I(Z)^{l}\setminus I(Z)^{l+1}\}.$$

Suppose now that $\G$ is a formal invariant curve of $X$ not
contained in the singular locus of $X$. A germ of holomorphic
map $\phi:(\C^n,0)\to(\C^n,0)$ will be called a {\em permissible
transformation for the pair $(X,\G)$} if it is of one of the
following types:

\begin{enumerate}
\item[1.] The germ of a holomorphic diffeomorphism.
\item[2.]  Let $Z$ be a germ of non-singular analytic submanifold at
$0\in\C^n$,
  transversal to $\G$ at $0$,
  invariant for $X$ and such that it satisfies $\nu_Z(X)\geq\nu_0(X)$.
  Let $\pi_Z:M\to U$ be the blow-up with center $Z$ and
  let $p\in\pi_Z^{-1}(0)$ be the point corresponding to the tangent of
  $\G$.
  Then there is
  an analytic chart $\tau$ of $M$ at $p$
  so that $\phi$ is the germ of $\pi_Z\tau^{-1}$ at $0\in\C^n$. We will
  say that $Z$ is a {\em permissible center} and that $\phi$ is a {\em permissible blow-up}.
    \item[3.] The curve $\G$ is non-singular, there are
    analytic coordinates $\zz=(z_1,...,z_n)$ at $0\in\C^n$
     such that $Z=\{z_1=0\}$ is invariant for $X$ and transversal to $\G$ and $\phi$ is the
    map
    $
    \phi(\zz)=(z_1^q,z_2,...,z_n)
    $ for some $q\in\N_{>0}$. We will
  say that $\phi$ is a {\em permissible $q$-ramification} (with respect to $Z$).
\end{enumerate}
In the last two cases, the non-singular hypersurface
$E_\phi=\phi^{-1}(Z)$ is called the {\em exceptional divisor of
$\phi$}. For convenience, $E_\phi=\emptyset$ in the case where
$\phi$ is a diffeomorphism. Notice that a permissible transformation $\phi$ is a local
diffeomorphism at every point in the complement of $E_\phi$.

Observe that the notion of permissible center depends essentially on $X$ (except for the transversality with $\G$); however, the notion of permissible blow-up depends on both $X$ and $\G$, since it chooses the point in the divisor corresponding to the tangent of $\G$. Note that $Z=\{0\}$ is a permissible center for $(X,\G)$  as long as $0$ is a singular point of $X$.

\begin{proposition}\label{pro:admissible-blow-up-X}
Let $\phi:\Cd n\to\Cd n$ be a permissible transformation for
$(X,\G)$. There exist a unique formal curve $\wt{\G}$ at $0\in\C^n$
such that $\phi^*\G\subset\wt{\G}$ (where
$\phi^*\G=\{g\circ\phi\,:\,g\in\G\}$) and a unique formal vector field
$\wt{X}$ at $0\in\C^n$ such that $\phi_\ast\wt{X}=X$, which
also has $\wt{\G}$ as an invariant curve and satisfies
$\nu_0(\wt{X})\geq\nu_0(X)$.
\end{proposition}

We will call $\wt{X}$ and $\wt{\G}$ the {\em transforms} of $X$ and $\G$ by $\phi$, respectively.

\begin{proof}
The case where $\phi$ is a germ of a diffeomorphism is clear.

Suppose that $\phi$ is a permissible blow-up with center $Z$. Take
analytic coordinates $\zz=(z_1,z_2,...,z_n)$ so that the tangent of
$\G$ corresponds to the $z_1$-axis and so that
$Z=\{z_1=z_2=\cdots=z_t=0\}$ where $t=\codim Z$ (thus $I(Z)$ is
generated by $z_1,...,z_t$). Then we may write $\phi:\Cd n\to\Cd n$
as
$$
\phi(\zz)=(z_1,z_1z_2,...,z_1z_t,z_{t+1},...,z_n).
$$
Let $\g(s)=(\g_1(s),...,\g_n(s))\in\C[[s]]^n$ be a parametrization
of $\G$ in the coordinates $\zz$. Then
$$
m=\nu_0(\G)=\nu(\g_1(s))<\nu(\g_j(s))\mbox{ for }j=2,...,n,
$$
where $\nu$ denotes the order in $s$. Also,
\begin{equation}\label{eq:gamma-tilde}
\wt{\g}(s)=
\Bigl(\g_1(s),\frac{\g_2(s)}{\g_1(s)},...,\frac{\g_t(s)}{\g_1(s)},\g_{t+1}(s),...,\g_n(s)\Bigr)
\in\C[[s]]^n
\end{equation}
is a parametrization of the formal curve $\wt{\G}$ which satisfies
$\phi^*\G\subset\wt{\G}$. The uniqueness of $\wt{\G}$ can be seen as follows: if $\bar{\g}(s)=(\bar{\g}_1(s),\bar{\g}_2(s),...,\bar{\g}_n(s))$ is a parametrization of another formal curve $\bar{\G}$ satisfying $\phi^*\G\subset\bar{\G}$ then we will have that $\phi\circ\bar{\g}(s)$ is another parametrization of $\G$ and necessarily $\phi\circ\bar{\g}(s)=\g(\sigma(s))$ where $\sigma(s)\in\C[[s]]$. Using the expression of $\phi$ and equation (\ref{eq:gamma-tilde}) one shows that also $\bar{\g}(s)=\wt{\g}(\sigma(s))$ and we are done.

Write $X=\sum_{i=1}^n a_i(\zz)\frac{\partial}{\partial
z_i}$. Since $\G$ is invariant and not contained in the singular locus of $X$, $X|_{\g(s)}\neq 0$ is a multiple of
$\g'(s)$ (see section~\ref{sec:definitions}) and hence
\begin{equation}\label{eq:orders-aj}
\nu(a_1(\g(s)))<\nu(a_j(\g(s)))\mbox{ for }j=2,...,n.
\end{equation}
On the other hand, the condition of $Z$ being invariant implies
that, for $j=1,...,t$, $a_j(\zz)\in I(Z)$ and hence $a_j(\phi(\zz))$ is divisible by
$z_1$. Using this property, the vector field
$\wt{X}=\sum_{i=1}^n \tilde{a}_i(\zz)\frac{\partial}{\partial
z_i}$ defined by
$$
\left\{
  \begin{array}{ll}
    \tilde{a}_j(\zz)=\dfrac{a_j(\phi(\zz))-z_ja_1(\phi(\zz))}{z_1}, &
 \hbox{for $j=2,...,t$;} \\[7pt]
    \tilde{a}_j(\zz)=a_j(\phi(\zz)), & \hbox{for $j\in\{1,t+1,...,n\}$,}
  \end{array}
\right.
$$
is formal and satisfies $\phi_\ast\wt{X}=X$. Let us show that $\nu_0(\wt{X})\geq\nu_0(X)$. Denote by $r=\nu_0(X)$. Since $Z$ is
a permissible center, we have $r\leq\nu_Z(X)$ and thus
$X(z_j)=a_j(\zz)\in I(Z)^r$ for $j=1,...,t$. Then, according to the
expression of the transformed vector field $\wt{X}$, it suffices
to show that for $j\in\{2,...,t\}$ the component $a_j(\zz)$ does not
contain a monomial of the form $c z_1^r$ with $c\neq 0$. Suppose,
on the contrary, that such a monomial appears in $a_{j}(\zz)$ for some
$j\in\{2,...,t\}$. Then we will have that the series
$a_j(\g(s))\in\C[[s]]$ has order $rm$. But since
$\nu(a_l(\g(s)))\geq \nu(a_l)\nu_0(\G)\geq rm$ for any
$l\in\{1,...,n\}$, this is a contradiction with
(\ref{eq:orders-aj}).

Assume now that $\phi$ is a permissible $q$-ramification,
written in some coordinates $\zz$ as
$\phi(\zz)=(z_1^q,z_2,...,z_n)$. Consider a parametrization of
$\G$ of the form $\g(s)=(s,\g_2(s),...,\g_n(s))$ in these
coordinates (recall that, from the definition of permissible
ramification, $\G$ is non-singular). Then
 $$
\wt{\g}(s)=(s,\g_2(s^q),...,\g_n(s^q))\in\C[[s]]^n
$$
is a parametrization of a formal curve $\wt{\G}$ satisfying
$\phi^*\G\subset\wt{\G}$. Uniqueness of $\wt{\G}$ comes from the property of $\G$ being non-singular: $\G$ is generated by the series $z_j-\g_j(z_1)$ for $j=2,...,n$ and thus, if $\phi^*\G\subset\wt{\G}$, then $\wt{\G}$ must be the (prime) ideal generated by $z_j-\g_j(z_1^q)$ for $j=2,...,n$.

On the other hand, being $Z=\{z_1=0\}$
invariant for $X$, if we write $X=\sum_{i=1}^n
a_i(\zz)\frac{\partial}{\partial z_i}$ then we have $
a_1(\zz)=z_1\bar{a}_1(\zz)$, where $\bar{a}_1(\zz)$ is a formal
series. We can check that the formal vector field $\wt{X}$ defined
by
$$
\wt{X}=\frac{z_1 \bar{a}_1(\phi(\zz))}{q}\frac{\partial}{\partial
z_1}+\sum_{i=2}^n a_i(\phi(\zz))\frac{\partial}{\partial z_i}
$$
satisfies $\phi_\ast\wt{X}=X$ and, clearly, $\nu_0(\wt{X})\geq\nu_0(X)$.
\end{proof}

It is worth to notice that
Proposition~\ref{pro:admissible-blow-up-X} remains true, except
for the uniqueness of the curve $\wt{\G}$ satisfying
$\phi^*\G\subset\wt{\G}$, if the condition of $\G$ being
non-singular in the definition of permissible ramification is
removed (consider, for example, $\G=(y^2-x^3)$ and $\phi(x,y)=(x^2,y)$).

\begin{remark}\label{rk:finite-jets}
Observe that finite truncations of the expression of the transform of a vector field by a permissible transformation are \emph{finitely determined}: if $\phi$ is a permissible transformation for $(X,\G)$ and $N\in\N$, there exists $N'\in\N$ such that if $\phi$ is also ``permissible'' for another formal vector field $Y$ (considering just the conditions in the definition of permissible center concerning the vector field and not the formal invariant curve) and $J_{N'}Y=J_{N'}X$ then for the transforms $\wt{X},\wt{Y}$ of $X,Y$ we have $J_{N}\wt{Y}=J_{N}\wt{X}$.
\end{remark}

Consider now a diffeomorphism $F\in\dif1\Cd n$ with a formal invariant curve $\G$ not contained in the set of fixed points of $F$. A germ of holomorphic map $\phi:\Cd n\to\Cd n$ will be called a \emph{permissible transformation for $(F,\G)$} if it is a permissible transformation for $(\log F,\G)$. In the following proposition, we show that permissible
transformations are compatible with the operation of taking the
infinitesimal generator of $F$ and with the existence of parabolic
curves.

\begin{proposition}\label{pro:admissible-blowing-up}
Let $\phi:\Cd n\to\Cd n$ be a permissible transformation for
$(F,\G)$, and let $\wt{X}$ and $\wt{\G}$ be the transforms of $\log F$ and $\G$ by $\phi$. There exists a unique  diffeomorphism $\wt{F}\in\dif1\Cd n$ such that $\phi\circ\wt{F}=F\circ\phi$, which
also satisfies $\log \wt{F}=\wt{X}$. Moreover, if $\wt{\vp}:\Delta\to\C^n$ is a parabolic curve of
$\wt{F}$ asymptotic to $\wt{\G}$ such that $\wt{\vp}(\Delta)$ is
contained in the domain of $\phi$ and does not intersect $E_\phi$,
then $\vp=\phi\circ\wt{\vp}$ is a parabolic curve of $F$ asymptotic
to $\G$ (provided $\vp$ is injective).
\end{proposition}

The pair $(\wt{F},\wt{\G})$ is called the
{\em transform of $(F,\G)$ by $\phi$}.

\begin{proof}
If $\phi$ is a germ of a diffeomorphism, the result is clear.

Assume that $\phi$ is a permissible blow-up with center $Z$. Consider analytic coordinates $\zz=(z_1,z_2,...,z_n)$ so that $\G$ is tangent to the $z_1$-axis, $Z=\{z_1=z_2=\cdots=z_t=0\}$ and so that
$\phi(\zz)=(z_1,z_1z_2,...,z_1z_t,z_{t+1},...,z_n)$.
The condition $\phi\circ\wt{F}=F\circ\phi$ can be written as
$$
z_j\circ\wt{F}(\zz)=\left\{%
\begin{array}{ll}
    \dfrac{z_j\circ F (\phi(\zz))}{z_1\circ F (\phi(\zz))}, & \hbox{if $j=2,...,t$;}  \\[10pt]
    z_j\circ F(\phi(\zz)), & \hbox{if $j\in\{1,t+1,...,n\}$}. \\
\end{array}%
\right.
$$
As in the case of formal curves, the invariance of $Z$ for $\log F$ is equivalent to the condition
$g\circ F\in I(Z)$ for any $g\in I(Z)$ and therefore $z_j\circ F(\phi(\zz))$ is divisible by $z_1$ for $j=1,...,t$. Moreover, $z_1\circ F(\phi(\zz))=z_1(1+A(\zz))$ with $A(0)=0$ and we conclude that
$\wt{F}\in\dif1\Cd n$. In order to prove that $\log \wt{F}=\wt{X}$, observe that $\nu_0(\wt{X})\ge\nu_0(X)\geq2$ by Proposition~\ref{pro:admissible-blow-up-X}, where $X=\log F$. Then $\Exp\wt{X}$ is a formal tangent to identity
diffeomorphism at $0\in\C^n$. Moreover, since
$\wt{X}^j(g\circ\phi)=X^j(g)\circ\phi$ for any
$g\in\wh{\mathcal{O}}_n$ and any $j\geq 1$, using
(\ref{eq:exponential}) we have that $\Exp\wt{X}$ satisfies formally
$\phi\circ\Exp\wt{X}=F\circ\phi$ and hence $\Exp\wt{X}=\wt{F}$.

To finish the proof for permissible blow-ups, let
$\wt{\vp}:\Delta\to\C^n$ be a parabolic curve of $\wt{F}$ whose
image does not intersect $E_\phi$. Since $\phi\circ\wt{F}=F\circ\phi$, we deduce easily that
$\vp=\phi\circ\wt{\vp}$ is a a parabolic curve of $F$ (notice that
in this case $\phi$ is injective outside $E_\phi$). On the other
hand, if $\wt{\vp}$ is asymptotic to $\wt{\G}$, i.e.,
$\zz\circ\wt{\vp}$ has asymptotic expansion at $s=0$ equal to
$\wt{\g}(s)\in\C[[s]]^n$, where $\wt{\g}(s)$ is a
parametrization of $\wt{\G}$ in the given coordinates, then, being
$\g(s)=\phi\circ\wt{\g}(s)$ a parametrization of $\G$,
$\vp$ is asymptotic to $\G$.

Assume now that $\phi$ is a permissible $q$-ramification,
written in some coordinates $\zz$ as
$\phi(\zz)=(z_1^q,z_2,...,z_n)$. If $\phi\circ\wt{F}=F\circ\phi$, using the expression of
$\phi$ we obtain
$$
(z_1\circ\wt{F}(\zz))^q=z_1\circ
F(\phi(\zz)),\;\;z_j\circ\wt{F}(\zz)=
   z_j\circ F (\phi(\zz)),\,j=2,...,n,
$$
 and since $z_1\circ F(\zz)=z_1(1+A(\zz))$ with $A(0)=0$, by invariance of $\{z_1=0\}$, we conclude that
$\wt{F}\in\dif1\Cd n$. The claim concerning the infinitesimal generator and the one concerning parabolic curves are proved as in the case of permissible blow-ups.
\end{proof}

\begin{remark}
Observe that if $\phi$ is a permissible ramification
and $\wt{\vp}:\Delta\to\C^n$ is a parabolic curve, there is no
need for the composition $\vp=\phi\circ\wt{\vp}$ to be injective;
that is the reason of the hypothesis ``provided $\vp$ is
injective'' stated in the proposition.
\end{remark}

\section{Reduction to Ramis-Sibuya form}\label{sec:reduction}

In this section we show that for any $F\in\dif1\Cd n$ and any formal invariant curve $\G$ of $F$, the pair $(F,\G)$ can be reduced to Ramis-Sibuya form by permissible transformations.

Let $X$ be a formal vector field and $\G$ be a formal invariant curve of $X$. A {\em
sequence of permissible transformations for $(X,\G)$} is a
composition
$$
\Phi=\phi_l\circ\phi_{l-1}\circ\cdots\circ\phi_1:\Cd n\to\Cd n
$$
such that $\phi_1$ is a permissible transformation for $(X,\G)$
and, for $j=1,...,l-1$, $(X_j,\G_j)$ is the transform of
$(X_{j-1},\G_{j-1})$ by $\phi_j$ and $\phi_{j+1}$ is a permissible
transformation for $(X_j,\G_j)$. The last pair
$(X_l,\G_l)$ will be called the {\em transform of $(X,\G)$ by the
sequence $\Phi$}. We also define the {\em total divisor of
$\Phi$} as the set
$E_\Phi=(\phi_l\circ\phi_{l-1}\circ\cdots\circ\phi_2)^{-1}(E_{\phi_1})$,
which is a normal crossing divisor at $0\in\C^n$.  If $F\in\dif1\Cd n$ and $\G$ is a formal invariant curve, we define the same concept of {\em sequence of permissible transformations} for $(F,\G)$ and the corresponding {\em transform} $(\wt{F},\wt{\G})$ using $X=\log F$.

\begin{theorem}[Reduction of a diffeomorphism to Ramis-Sibuya form]\label{th:reduction-singularities-F}
Let $F\in\dif1\Cd n$, let $\G$ be a formal invariant curve of
$F$ and assume that $\G$ is not contained in the set of fixed points
of $F$. There is sequence $\Phi$ of permissible transformations for
$(F,\G)$ so that the transform $(\wt{F},\wt{\G})$ of $(F,\G)$ by
$\Phi$ is in Ramis-Sibuya  form and $\Fix(\wt{F})=E_\Phi$.
\end{theorem}

The corresponding statement for vector fields is the following one.

\begin{theorem}[Reduction of a vector field to Ramis-Sibuya form]\label{th:reduction-singularities-X-bis}
Let $X$ be a formal vector field at $0\in\C^n$ with $X(0)=0$ and let
$\G$ be an invariant formal curve of $X$ not contained in the
singular locus of $X$. There exists a sequence
$\Phi$ of permissible transformations for $(X,\G)$ and there are
analytic coordinates $(x,\yy)=(x,y_2,...,y_n)$ at $0\in\C^n$ so
that the total divisor of $\Phi$ is
$E_\Phi=\{x=0\}$
and
 \begin{itemize}
 \item[(i)] The transform of $\G$ by $\Phi$ is non
 singular and tangent to the $x$-axis.
 \item[(ii)] The transform of $X$ by $\Phi$ is
 written as
     \begin{equation}\label{eq:X-RS}
     \Phi^*X=x^k\left[x^{p+1}u(x,\yy)\frac{\partial}{\partial x}+
     \left(c(x)+(\mc{D}(x)+x^p\mc{C}+O(x^{p+1}))\yy+O(\|\yy\|^2)\right)
     \frac{\partial}{\partial \yy}\right]
     \end{equation}
     where $u\in\C[[x,\yy]]$ with $u(0,0)\neq 0$, $c(x)\in\C[[x]]^{n-1}$ with $c(0)=0$, $k\geq 0$
      and either $p=0$ and $\mc{C}\neq0$ or $p\geq 1$ and $\mc{D}(x)$ a diagonal matrix of polynomials
      of degree at most $p-1$ commuting with $\mc{C}$ and $\mc{D}(0)\neq 0$.
\end{itemize}
Any such sequence $\Phi$ will be called a {\em reduction of $(X,\G)$ to RS-form}.
\end{theorem}

Theorem~\ref{th:reduction-singularities-F} is a consequence of Theorem~\ref{th:reduction-singularities-X-bis}  taking $X=\log F$ and using Proposition 5.3. Notice that in this case the number $k$ in equation (\ref{eq:X-RS}) is at least $1$ since $\nu_0(\Phi^*X)\geq 2$ by Proposition~\ref{pro:admissible-blow-up-X}.

Theorem~\ref{th:reduction-singularities-X-bis} is, in essence, quite well known
in the theory of reduction of singularities of vector fields as well
as in the theory of systems of meromorphic ODEs with irregular
singularity. It contains a weak form of ``local uniformization'' of
$X$ along $\G$, i.e. reduction to non-nilpotent linear part (in the
context of real vector fields in dimension three, see Cano et al.
\cite{Can-M-R} for a local uniformization along non-oscillating
trajectories and also Panazzolo \cite{Pan} for a global reduction of
singularities). The particular form of equation (\ref{eq:X-RS})
requires more than just a non-nilpotent linear part of the vector
field. In our case (reduction {\em along a formal invariant curve}),
that particular form is obtained, once we associate to $X$ a system
of $n-1$ meromorphic ODEs after some initial blow-ups, from
classical results in the theory of ODEs: see Turrittin \cite{Tur}
(for systems of linear ODEs), Braaksma \cite{Bra} (for the non-linear case) or Cano et al.
\cite{Can-Mou-San1,Can-Mou-San2} (for related statements for real
vector fields in dimension three). However, we could not find an
statement with the precise terms of
Theorem~\ref{th:reduction-singularities-X-bis} needed for the
purposes of this article. Hence, for the sake of completeness, we
provide here a self-contained proof to which we devote the rest of this section.

\strut

Let us first describe the situation after a
punctual blow-up. Let $(x,\yy)$ be adapted coordinates for $\G$ and
let $\phi:\Cd n\to\Cd n$ be a permissible blow-up for $(X,\G)$ at
the origin. There is a constant vector $\xi\in\C^{n-1}$ so that
$\phi$ is written, for some coordinates $(x,\wt{\yy})$, as
\begin{equation}\label{eq:point-blow-up-coordinates}
(x,\yy)=\phi(x,\wt{\yy})=(x,x(\wt{\yy}-\xi)).
\end{equation}
The total
transform of $X$ by $\phi$ can be written as $\wt{X}=x^{\nu_0(X)-1}\bar{X}$
where $x$ is a coordinate such that $E_\phi=\{x=0\}$ and $\bar{X}$
is a formal vector field singular at $0$ (in fact
$\nu_0(\wt{X})\geq\nu_0(X)$, see
Proposition~\ref{pro:admissible-blow-up-X}). We deduce that $\wt{X}$ is again
singular at the origin, which is then a permissible center for the transform
$(\wt{X},\wt{\G})$ of $(X,\G)$. Performing in this way the blow-ups
at the infinitely near points of $\G$ and using resolution of
singularities of curves, we can assume that $\G$
is non-singular and that there is a system of adapted coordinates
$(x,\yy)$ for $\G$ such that $X=x^{e}\bar{X}$, where $\bar{X}$ is not divisible by $x$ and $e\geq 1$ if
$\nu_0(X)\geq 2$ (and thus $\{x=0\}$ is invariant for $X$). Let
$\g(x)=(x,\bar{\g}(x))\in\C[[x]]^n$ be a parametrization of $\G$ in
these coordinates and write
$$
\bar{X}=a(x,\yy)\frac{\partial}{\partial x}+
 {\bf b}(x,\yy)\frac{\partial}{\partial\yy},
$$
where $a(x,\yy)\in\C[[x,\yy]]$. By invariance of $\G$, $X|_{\g(x)}$
is collinear to the vector $(1,\bar{\g}'(x))$ and then $a(\g(x))\neq 0$, since
$\G$ is not contained in the singular locus of $X$.

We may suppose that $\bar{X}$ is singular at the origin. In fact, if $\bar{X}(0)\neq 0$ we must have $e\geq 1$
and, since in this case $\G$ is the unique formal solution of
$\bar{X}$ at $0$ and it is transversal to $\{x=0\}$, we must have $a(0)=a_0\neq 0$. After a new blow-up at the origin, and
taking coordinates as in (\ref{eq:point-blow-up-coordinates}), the
transform of $X$ is written as
$$
\wt{X}=x^{e-1}\left[x\left(a_0+O(x)\right)\frac{\partial}{\partial
x}+\left(-a_0I_{n-1}\wt{\yy}+O(x)\right)\frac{\partial}{\partial\wt{\yy}}\right]
$$
which is of the required form (\ref{eq:X-RS}) with $k=e-1\geq 0$ and
$p=0$.

Assume then that $\bar{X}$ is singular. Let $r$ be the multiplicity
of the restriction $\bar{X}|_\G$ (notice that $r$ is the order of
the series $a(\g(x))$ and hence $1\leq r<\infty$). Let $\phi$ be the
permissible blow-up at the origin written in coordinates as in
(\ref{eq:point-blow-up-coordinates}), where in this case
$\xi=\bar{\g}'(0)$. The effect of this transformation on the vector
field is that the exponent of $x$ increases in any monomial of
$a(x,\yy)$ with positive degree in the $\yy$-variables and in any
monomial of the components of ${\bf b}(x,\yy)$ with degree at least
two in the $\yy$-variables. Thus, after several such blow-ups of
points, the new $a$-component of the transform of $\bar{X}$ is of
the form $x^ru(x,\yy)$ where $u(0,0)\neq 0$ and each monomial in the
new ${\bf b}$-components of degree at least two in the
$\yy$-variables is divisible by $x^r$. Moreover, the monomials of
degree 0 in the $\yy$-variables of the new ${\bf b}$-components can
be assumed to be divisible by $x^r$ by means of a translation of the
form $\wt{\yy}=\yy+Q(x)$, where $Q(x)$ is a polynomial truncation of
$\bar{\g}(x)\in\C[[x]]^{n-1}$.

We conclude that, up to permissible blow-ups of points and changes
of coordinates, the vector field $X$ may be written as
\begin{equation}\label{eq:systemEDOs-1}
X=x^l\left[x^{q+1}u(x,\yy)\frac{\partial}{\partial x}+(c(x)+A(x)\yy+
\Theta(x,\yy))\frac{\partial}{\partial\yy}\right]
\end{equation}
where  $l\geq 0, q\geq -1$, $u(0,0)\neq 0$, $c(0)=0$
and $\Theta\in\C[[x,\yy]]^{n-1}$ has order at least 2 in the
$\yy$-variables. We may assume $q\geq 0$ since otherwise
the vector field $x^{-l}X$ would be non-singular, a case already
treated above.  Moreover, up to increasing $l$ and decreasing $q$, we may assume also that $A(0)\neq 0$.

Notice that if $q=0$ then $X$ is already in the required form
(\ref{eq:X-RS}) with $l=k$ and $p=0$. We assume that $q\geq 1$. To
the vector field $X$ in (\ref{eq:systemEDOs-1}) we can associate
the system of $n-1$ formal meromorphic ODEs
\begin{equation}\label{eq:systemEDOs-2}
x^{q+1}\yy'=u(x,\yy)^{-1}\left(c(x)+A(x)\yy+ \Theta(x,\yy) \right).
\end{equation}
Systems of meromorphic ODEs are extensively
studied in the linear case
$$
x^{q+1}\yy'=B(x)\yy,
$$
where $q\geq-1$ and $B(x)$ is a matrix with formal coefficients. For such a system, the number
$q$ is called
 the {\em Poincar\'{e}
rank} if $B(0)\neq 0$; in this case, the origin is singular if $q\geq 0$, a {\em regular singularity} if $q=0$
and an {\em irregular singularity} if $q\geq1$.
We will use the following particular
result of this theory (see Turrittin \cite{Tur} or also Wasow
\cite{Was} or Balser \cite{Bal}).
\begin{theorem}[Turrittin]\label{th:turritin}
Consider a formal linear $m$-dimensional system
$
x^{q+1}\yy'=B(x)\yy,
$
where $q$ is the Poincar\'{e} rank, and assume that $q\geq 1$. Then,
after a finite number of transformations of the variables among the
following types
\begin{enumerate}[$\bullet$]
\item Polynomial linear transformation
  $$
 L(P(x)):\;\;\yy=P(x)\wt{\yy},\;
 P(x)\in\mc{M}_{m}(\C[x])\mbox{ with }P(0)\mbox{ invertible}.
$$
  \item Shearing transformation
  $$
 S(k_1,...,k_m):\;\;
\yy=\diag(x^{k_1},...,x^{k_m})\wt{\yy},\;k_j\in\N_{\geq 0}.
  $$
  \item Ramification
  $$
R(\alpha):\;\; x=\wt{x}^\alpha,\; \alpha\in\N_{>0}.
  $$
\end{enumerate}
the system transforms into a system
$$x^{p+1}\wt{\yy}'=(\mc{D}(x)+x^p\mc{C}+O(x^{p+1}))\wt{\yy},
$$ where either $p=0$ and $\mc{C}\neq0$ or $p\geq 1$, $\mc{D}(x)$ is a diagonal matrix of
polynomials
      of degree at most $p-1$ commuting with $\mc{C}$ and $\mc{D}(0)\neq 0$.
\end{theorem}
Polynomial linear transformations, shearing transformations and ramifications, as defined in Turrittin's Theorem, will be called {\em T-transformations}.

\begin{remark}\label{rk:shearing}
Notice that a polynomial linear transformation does not change the Poincar\'{e} rank and that a ramification $R(\alpha)$ multiplies it by $\alpha$. The effect of a shearing transformation $S(k_1,...,k_m)$ on the Poincar\'{e} rank depends on the parameters $k_1,...,k_m$. Looking carefully at the proof of Theorem~\ref{th:turritin} (see for example \cite{Was}, section 19), we may observe that the shearing transformations are chosen so that their application never makes the Poincar\'{e} rank increase strictly.
\end{remark}

We resume the proof of Theorem~\ref{th:reduction-singularities-X-bis}. Assume that $X$ is written as
in (\ref{eq:systemEDOs-1}). Consider
the formal change of variables $ \yy=\hat{\yy}+\bar{\g}(x)$, for
which $\G=\{\hat{\yy}=0\}$, and write $X$ in the variables
$(x,\hat{\yy})$ as
$$x^l\left[x^{q+1}u(x,\hat{\yy}+\bar{\g}(x))\frac{\partial}{\partial
x}+(\hat{A}(x)\hat{\yy}+
\hat{\Theta}(x,\hat{\yy}))\frac{\partial}{\partial\hat{\yy}}\right]$$
where $\hat{A}(0)\neq 0$ and $\hat{\Theta}(x,\hat{\yy})=O(\|\hat{\yy}\|^2)$.
The system (\ref{eq:systemEDOs-2}) becomes
$$x^{q+1}\hat{\yy}'=u(x,\hat{\yy}+\bar{\g}(x))^{-1}\left(\hat{A}(x)\hat{\yy}+
\hat{\Theta}(x,\hat{\yy})\right).
$$
Apply Theorem~\ref{th:turritin} to the linear system
$x^{q+1}\ww'=u(x,\bar{\g}(x))^{-1}\hat{A}(x)\ww$ with $m=n-1$,
associated to the vector field
$$Y=x^l\left[x^{q+1}u(x,\bar{\g}(x))\frac{\partial}{\partial
x}+\hat{A}(x)\ww\frac{\partial}{\partial\ww}\right].$$ Let us
justify that the T-transformations involved are sequences of
permissible transformations for $(Y,\{\ww=0\})$. This is clear for
polynomial linear transformations and for ramifications. On the
other hand, a shearing transformation can be viewed as a composition
of blow-ups with centers of codimension two. More precisely,
$S(k_2,...,k_{n})$ is obtained by blowing-up $k_2$ times the center
with equations $\{x=w_2=0\}$ (considering repeatedly the expression $\phi(x,\ww)=(x,xw_2,w_3,...,w_n)$),
followed by blowing-up $k_3$ times the center $\{x=w_3=0\}$, and so
on. It suffices to show that such local blow-ups are permissible. Consider one of such centers, for instance $Z=\{x=w_2=0\}$, and let $\phi(x,\ww)=(x,xw_2,w_3,...,w_n)$ be the blow-up with center $Z$. One can check that the
transform of $x^{-l}Y$ by $\phi$, which is a meromorphic vector
field, has coefficients with no poles if and only if $x^{-l}Y(w_2)$
belongs to the ideal $I(Z)=(x,w_2)$, i.e. if and only if $Z$ is
invariant for $x^{-l}Y$. Observe that if the resulting vector field
has poles then the Poincar\'{e} rank of the associated linear system
 will increase and so, by Remark~\ref{rk:shearing}, we conclude that $Z$
 is invariant for $x^{-l}Y$ and thus for $Y$. Moreover, $Y(x)=O(x^{l+1})$ and $Y(w_2)=x^l(\sum_{i=2}^n\lambda_i w_i+O(x))$,
where $\lambda_i\in\C$ and, being $Z$ invariant for $x^{-l}Y$,
$\lambda_i=0$ if $i\neq 2$. We conclude that $\nu_Z(Y)\geq
l+1=\nu_0(Y)$ and hence $Z$ is a permissible center for
$(Y,\{\ww=0\})$.

Denote by $\Psi:(x,\wt\ww)\mapsto(x^{\beta},\Psi_2(x,\wt\ww))$ the
composition of the T-transformations used in Turrittin's process for
$Y$. Then, taking into account the conclusion in
Theorem~\ref{th:turritin}, $\Psi$ is a reduction of $(Y,\{\ww=0\})$
to RS-form. More precisely,
\begin{equation}\label{eq:turrittin-to-linear}
\Psi^*Y=x^k\left[x^{p+1}\beta^{-1}u(x^{\beta},\bar{\g}(x^{\beta}))\frac{\partial}{\partial
x}+\left(\mc{D}(x)+x^p\mc{C}+O(x^{p+1})\right)\wt{\ww}\frac{\partial}{\partial\wt{\ww}}\right]
\end{equation}
with the conditions for $k$, $p$, $\mc{D}(x)$ and $\mc{C}$ stated in
Theorem~\ref{th:reduction-singularities-X-bis}. Observe that $k+p=\beta(l+q)$.

 For any $m$, consider the diffeomorphism $\phi_m:(\C^n,0)\to(\C^n,0)$
 given by the change of variables $\yy=\wt{\yy}+J_m\bar{\g}(x)$.
 The transform $X_m$
 of $X$ by $\phi_m$ is written as
$$
X_m=x^l\left[x^{q+1}u(x,\wt{\yy}+J_m\bar{\g}(x))\frac{\partial}{\partial
x}+\left(c_m(x)+A_m(x)\wt{\yy}+
\Theta_m(x,\wt{\yy})\right)\frac{\partial}{\partial\wt{\yy}}\right],
$$
where $A_m(x),\Theta_m(x,\wt{\yy})$ converge to $\hat{A}(x),
\hat\Theta(x,\wt\yy)$ respectively in the Krull topology and
$\Theta_m(x,\wt{\yy})=O(\|\wt{\yy}\|^2)$. Moreover, the transform
$\G_m$ of $\G$ has order of contact at least $m$ with the $x$-axis
$\{\wt{\yy}=0\}$ and thus the order of $c_m(x)$ is at least $m$ (see
Remark~\ref{rk:RS-form}). On the other hand, consider for $M\in\N$
the map $\psi_M:(x,\ww)\mapsto(x,\tilde{\yy}=x^M\ww)$. Notice that
$\psi_M$ is a composition of punctual permissible blow-ups for
$(X_m,\G_m)$ if $m>M+1$. Let $(X_{m,M},\G_{m,M})$ be the transform
of $(X,\G)$ by $\psi_M\circ\phi_m$, when $m>M+1$. Taking into
account the effect, mentioned above, of punctual blow-ups in the
monomials of positive order in the $\yy$-variables, it is not
difficult to check that, if $M$ is sufficiently big, then the map
$\Psi$ is a sequence of permissible transformations, not only for
$(Y,\{\ww=0\})$, but also for $(X_{m,M},\G_{m,M})$ (the centers
involved have equations of the form $\{x=w_j=0\}$ and the property
of being permissible for $(X_{m,M},\G_{m,M})$, by the same arguments
used above for $(Y,\{\ww=0\})$, only depends on the linear part in
each step as long as $x$ divides the non-linear terms).

In order to finish the proof of
Theorem~\ref{th:reduction-singularities-X-bis}, it suffices to prove
that if $m$ and $M$ are sufficiently big then $\Psi$ is a reduction
of $(X_{m,M},\G_{m,M})$ to RS-form: the map
$\Psi\circ\psi_M\circ\phi_m$ will then be a reduction of $(X,\G)$ to
RS-form. We use Remark~\ref{rk:finite-jets} repeatedly for any of
the permissible transformations of the composition $\Psi$: there
exists $N'$ such that, given a vector field $W$, if $J_{N'}(W)=J_{N'}(Y)$
and $\Psi$ is a sequence of permissible transformations for
$W$ (with no conditions concerning any formal invariant curve) then
\begin{equation}\label{eq:jet-Z}
J_{k+p+1}(\Psi^*W)=
J_{k+p+1}(\Psi^*Y).
\end{equation}
Now, choose $m, M$ sufficiently big with $m>M+1$ so that
$J_{N'}(X_{m,M})=J_{N'}(\psi_M^*Y)$. Notice that $\psi_M$ is a
sequence of permissible blow-ups for $(Y,\{\ww=0\})$ and that
$$
\psi_M^*Y=x^l\left[x^{q+1}u(x,\bar{\g}(x))\frac{\partial}{\partial
x}+\left(\hat{A}(x)-Mx^qu(x,\bar{\g}(x))I_{n-1}\right)\ww\frac{\partial}{\partial\ww}\right].$$
Hence, if we put
$W_M=Mx^{l+q}u(x,\bar{\g}(x))I_{n-1}\ww\frac{\partial}{\partial\ww}$
then $J_{N'}(X_{m,M}+W_M)=J_{N'}(Y)$ and, by (\ref{eq:jet-Z}) and
(\ref{eq:turrittin-to-linear}), and since $\Psi$ is also permissible
for $X_{m,M}+W_M$, we have that $\Psi^*(X_{m,M}+W_M)$ has the desired
form (\ref{eq:X-RS}) of
Theorem~\ref{th:reduction-singularities-X-bis} with
$k+p=\beta(l+q)$. Notice finally, using the expressions (see
Theorem~\ref{th:turritin}) of the T-transformations involved in the
map $\Psi$, that
$$
\Psi^*\left(W_M\right)=
x^{\beta(l+q)}v(x)I_{n-1}\wt{\ww}\frac{\partial}{\partial\wt{\ww}},
$$
where $v(0)\neq 0$. We conclude that $\Psi^*(X_{m,M})$ is also in the
form (\ref{eq:X-RS}) and thus $\Psi$ is a reduction of
$(X_{m,M},\G_{m,M})$ to RS-form as wanted  (we need to remark that it is possible that $p=0$ and $\mc{C}=-v(0)I_{n-1}$ in which case the resulting matrix $\mc{C}$ for $\Psi^*(X_{m,M})$ vanishes; however, in this case, $x^{-(k+1)}\Psi^*(X_{m,M})$ is non singular and can be treated as we did above).
Theorem~\ref{th:reduction-singularities-X-bis} is finished.

\section{Conclusion}

In this final section, we restate and prove Theorem~\ref{Theorem2}
in the introduction.

 Given a diffeomorphism $F$ with an invariant
formal curve $\G$ not contained in the set of fixed points, we say
that $\G$ is {\em well placed for $F$} if there exists a sequence
$\Phi$ that reduces $(F,\G)$ to Ramis-Sibuya form as in
Theorem~\ref{th:reduction-singularities-F} and such that one of the
attracting directions of the transform $(\wt{F},\wt{\G})$ is
contained in the saddle domain.

\begin{theorem} Let $F\in\dif1\Cd n$ have an invariant formal curve $\G$ not
contained in $\Fix(F)$. If $\Gamma$ is well placed for $F$, then there
exists at least one attracting parabolic curve for $F$ asymptotic to
$\Gamma$.
\end{theorem}
\begin{proof}
By Theorem~\ref{th:main-technical}, there exists an attracting
parabolic curve $\wt{\vp}$ of $\wt{F}$ asymptotic to $\wt{\G}$. This
parabolic curve does not cut the total divisor $E_\Phi$ of $\Phi$,
since $E_\Phi=\Fix(\wt{F})$ by
Theorem~\ref{th:reduction-singularities-F}. Hence, if
$\Phi=\phi_l\circ\phi_{l-1}\circ\cdots\circ\phi_1$ and
$\Phi_j=\phi_l\circ\phi_{l-1}\circ\cdots\circ\phi_{l-j}$ for
$j=0,...,l-2$, the map $\vp_j=\Phi_j\circ\wt{\vp}$ does not cut
either the exceptional divisor $E_{\phi_{l-j-1}}$ of $\phi_{l-j-1}$.
Applying recursively Proposition~\ref{pro:admissible-blowing-up} for
$j=0,...,l-2$, the map $\vp=\Phi\circ\wt{\vp}$ is a parabolic curve
of $F$ asymptotic to $\G$ (notice that we may assume that $\vp$ is
injective restricting the domain of $\wt{\vp}$ by part (ii) of
Theorem~\ref{th:main-technical}).
\end{proof}
Notice finally that the remark mentioned in the introduction about finite determinacy of the property of being well placed is a consequence of Remark~\ref{rk:finite-jets}.


\end{document}